\documentclass{amsart}
\usepackage{tikz,hyperref,amsbsy,cite}

\usepackage{enumerate}
\newtheorem{theorem}{Theorem}[section]
\newtheorem{proposition}[theorem]{Proposition}
\newtheorem{lemma}[theorem]{Lemma}
\theoremstyle{definition}
\newtheorem{example}[theorem]{Example}
\numberwithin{equation}{section}

\newcommand{\ba}{\mathbf{a}}
\newcommand{\bb}{\mathbf{b}}

\newcommand{\bu}{\mathbf{u}}
\newcommand{\bv}{\mathbf{v}}

\newcommand{\cG}{\mathcal{G}}
\newcommand{\cK}{\mathcal{K}}
\newcommand{\cU}{\mathcal{U}}

\title[Hausdorff dimension of double base expansions]{Hausdorff dimension of double base expansions and binary shifts with a hole}
\author[J. Lu] {Jian Lu}
\address[J. Lu \& Y. Zou]{School of Mathematical Sciences, Shenzhen University, 518061, P.R.China }
\email{jianlu@szu.edu.cn, yuruzou@szu.edu.cn}
\author[W. Steiner]{Wolfgang Steiner}
\address[W. Steiner]{Universit\'{e} Paris Cit\'{e}, CNRS, IRIF, F--75013 Paris, France}
\email{steiner@irif.fr}
\author[Y. Zou]{Yuru Zou}

\keywords{Double base expansions, univoque set, survivor set, topological entropy, Hausdorff dimension, open dynamical system}

\subjclass{11A63, 28A78, 11B83, 37B10, 37B40, 68R15}

\thanks{The second author was supported by the ERC grant DynAMiCs (101167561) of the European Research Council, the bilateral grant SYMDYNAR (ANR-23-CE40-0024 and FWF I 6750) of the Agence Nationale de la Recherche and the Austrian Science Fund, and by the ANR project IZES (ANR-22-CE40-0011). The third author was supported by the National Natural Science Foundation of China (NSFC) 12571098, Shenzhen Basis Research Project JCYJ20210324094006017, Natural Science Foundation of
Guangdong Province of China 2023A1515011691. }

\begin{document}

\begin{abstract}
For two real bases $q_0, q_1 > 1$, a binary sequence $i_1 i_2 \cdots \in \{0,1\}^\infty$ is the $(q_0,q_1)$-expansion of the number
\[
\pi_{q_0,q_1}(i_1 i_2 \cdots) = \sum_{k=1}^\infty \frac{i_k}{q_{i_1} \cdots q_{i_k}}.
\]
Let $\mathcal{U}_{q_0,q_1}$ be the set of all real numbers having a unique $(q_0,q_1)$-expansion. 
When the bases are equal, i.e., $q_0 = q_1 = q$, Allaart and Kong (2019) established the continuity in $q$ of the Hausdorff dimension of the univoque set~$\mathcal{U}_{q,q}$, building on the work of Komornik, Kong, and Li (2017). 
We derive explicit formulas for the Hausdorff dimension of $\mathcal{U}_{q_0,q_1}$ and the entropy of the underlying subshift for arbitrary $q_0, q_1 > 1$, and prove the continuity of these quantities as functions of $(q_0,q_1)$. 
Our results also concern general dynamical systems described by binary shifts with a hole, including, in particular, the doubling map with a hole and (linear) Lorenz maps.
\end{abstract}

\maketitle

\section{Introduction}
For a pair of real bases $q_0, q_1 > 1$, a sequence $i_1i_2\cdots \in \{0,1\}^\infty$ is called a \emph{$(q_0,q_1)$-expansion} of~$x$  if
\[
x = \pi_{q_0,q_1}(i_1i_2\cdots) := \sum_{k=1}^\infty \frac{i_k}{q_{i_1}\cdots q_{i_k}}.
\]
If $q_0=q_1=q$, then the sequence $i_1i_2\cdots$ is called a \emph{$q$-expansion} of~$x$. 
The set
\[
J_{q_0, q_1} := \pi_{q_0,q_1}\big(\{0,1\}^\infty\big)
\]
is a self-similar subset of $[0,\frac{1}{q_1-1}]$ generated by the \emph{iterated function system} (IFS)
\[
\{f_0(x) = x/q_0,\ f_1(x) = (x{+}1)/q_1\}.
\]
By \cite[Lemma~3.1]{KomSteZou2022}, we could replace the \emph{digits}~$0,1$ by real numbers $d_0, d_1$ satisfying $d_0(q_1{-}1) \ne d_1(q_0{-}1)$, i.e., consider the IFS $\{f_0(x) = \frac{x+d_0}{q_0},\, f_1(x) = \frac{x+d_0}{q_0}\}$, without changing our results.

For $q_0 {+} q_1 < q_0q_1$ , i.e., $\frac{1}{q_0} {+} \frac{1}{q_1} < 1$, the IFS satisfies the strong separation~condition, and then each $x\in J_{q_0, q_1}$ has a unique $(q_0, q_1)$-expansion. 
When $q_0 {+} q_1 = q_0q_1$, every $x \in J_{q_0, q_1}$ has a unique $(q_0, q_1)$-expansion except for a countable set of exceptions that have precisely two expansions. 
When $q_0 {+} q_1 > q_0q_1$, the situation is much more complicated. 
For instance, when $q_0 = q_1 = q \in (1,2)$, Lebesgue almost every $x \in J_{q,q}$ has a continuum of expansions; see \cite{DajdeV2007,Sid2003,Sid2007}. 
The study of representations of real numbers in non-integer bases (or simply $q$-expansions) has drawn significant attention since the seminal papers of R\'enyi~\cite{Ren1957} and Parry~\cite{Par1960}.
It has connections to many other areas of mathematics such as fractal geometry, ergodic theory, symbolic dynamics and number theory \cite{AlcBarBakKon2019,All2017,AllBakKong2019,AllKong2019,AllKong2021a,AllKong2021b,AllKong2023,AllKong2024,BakKon2020,BoyDecHal2016,Cla2016,DajdeV2007,DeVKom2009,KalKonLiLu2019,KalSte2012,KomKonLi2017,KomSteZou2022,KongLi2015}.

In the present paper, we are interested in the \emph{univoque set}
\[
\cU_{q_0, q_1} := \big\{x\in J_{q_0, q_1}: \#\pi_{q_0, q_1}^{-1}(x)=1\big\},
\]
for $q_0{+}q_1\geq q_0q_1$.
By definition, each $x\in \cU_{q_0,q_1}$ has a unique $(q_0, q_1)$-expansion. 
The univoque set is called trivial if $\cU_{q_0, q_1}=\{0, 1/(q_1{-}1)\}$.
For $q_0 = q_1 = q$, the univoque set $\cU_{q, q}$ was first properly studied by Erd\H{o}s et al.\ in the early 1990s \cite{ErdHorJoo1991,ErdJoo1992a,ErdJooKom1990}. 
Since then, many interesting properties of this set have been described \cite{KomLor1998, DeVKom2009, Bak2014, All2017,   AllKong2021a,AllKong2021b,ZouLiLuKom2021,DevKomLor2022}. 
In particular, the univoque set corresponds to the survivor set of a dynamical system with a hole.
This kind of open dynamical systems was first considered in the seminal paper by Pianigiani and Yorke~\cite{PiaYor1979} and has received a lot of attention in recent years \cite{Alc2014, GleSid2015, Cla2016, BakKon2020, KalKonLanLi2020, AllKong2023, AllKong2024}.

Let us recall remarkable results of Glendinning and Sidorov~\cite{GleSid2001}: If $1 < q \le \frac{1+\sqrt{5}}{2}$, then $\cU_{q,q}$ is trivial; if $\frac{1+\sqrt{5}}{2} < q < q_{\mathrm{KL}}$, then $\cU_{q,q}$ is countable infinite; if $q = q_{\mathrm{KL}}$, then $\cU_{q,q}$ is uncountable but of zero Hausdorff dimension; if $q > q_{\mathrm{KL}}$, then $\cU_{q,q}$ has positive Hausdorff dimension. Here, $q_{\mathrm{KL}}\approx 1.787$ denotes the Komornik--Loreti constant, i.e., the smallest base $q > 1$ where $x = 1$ has a unique $q$-expansion.
 
Recently, the Hausdorff dimension of the set~$\cU_{q,q}$ and the entropy of the coresponding set of sequences $U_{q,q} := \pi^{-1}_{q, q}(\cU_{q,q})$ have been the subject of a large number of research articles \cite{KongLi2015, KomKonLi2017, AlcBarBakKon2019,AllKong2019,KalKonLiLu2019,AllBakKong2019,KonLiLuWanXu2020,AllKong2021a,AllKong2021b}.
In particular, Komornik, Kong and Li~\cite{KomKonLi2017} established an explicit relation between the Hausdorff dimension of $\cU_{q,q}$ and the entropy of $U_{q, q}$; building on this result, Allaart and Kong \cite{AllKong2019} proved the continuity of this dimension function in~$q$.
We state these results only for the two digit case, even though they also hold for larger digit sets.
For $q\in (q_{KL}, 2]$, the Hausdorff dimension of the univoque set~$\cU_{q,q}$ is
\begin{equation*}\label{e:HausdorffD-KKL}
\dim_H{\cU_{q,q}} = \frac{h(U_{q,q})}{\log q},
\end{equation*}
where 
\[
h(U_{q,q}) = \lim_{n\to\infty} \frac{1}{n} \log \#\{i_1\cdots i_n \,:\, i_1i_2\cdots \in U_{q,q}\}
\]
is the \emph{topological entropy} of~$U_{q,q}$.
(Even though $U_{q,q}$ is usually not closed, we can use the formula for the topological entropy of subshifts because taking the closure only adds countable many elements, which does not change the entropy defined by Bowen~\cite{Bow1973}.)
The function $q \mapsto h(U_{q,q})$ is proved to be continuous on $(1, 2]$. 

Now, we return to the case of two distinct bases~$q_0,q_1$. 
Similar to the case of one base, we can also describe two thresholds: one separating bases with trivial and non-trivial univoque sets, called generalised golden ratio; the other one separating bases with countable and uncountable univoque
sets, called generalised Komornik–Loreti constant. 
More precisely, for $q_0 > 1$, define
\[
\begin{aligned}
\cG(q_0) & := \inf\big\{q_1 > 1 \,:\, \cU_{q_0, q_1} \neq \big\{0, \tfrac{1}{q_1-1}\big\}\big\}, \\
\cK(q_0) & := \inf \big\{q_1 > 1 \,:\, \cU_{q_0, q_1}\ \text{is uncountable}\big\}.
\end{aligned}
\]
Then for all $q_0>1$,
\begin{itemize}
\item 
$\cU_{q_0, q_1}$ is trivial when $q_1 < \cG(q_0)$;
\item 
$\cU_{q_0, q_1}$ is infinite when $q_1 > \cG(q_0)$;
\item 
$\dim_H\cU_{q_0, q_1} = 0$  when $q_1 \le \cK(q_0)$;
\item 
$\dim_H\cU_{q_0, q_1}>0$  when $q_1>\cK(q_0)$.	 
\end{itemize}
Together with Komornik~\cite{KomSteZou2022}, the second and third authors gave a complete characterisation of the functions~$\cG(q_0)$ and~$\cK(q_0)$, proved that the two functions are  symmetric in $q_0,q_1$, continuous and strictly decreasing, hence almost everywhere differentiable on $(1,\infty)$. 
The Hausdorff dimension of the set of~$q_0$ satisfying $\cG(q_0) = \cK(q_0)$  is zero.
The main goal of the present paper is to determine the exact Hausdorff dimension of $\cU_{q_0, q_1}$ when $q_1 > \cK(q_0)$ and to explore its continuity in~$q_0$ and~$q_1$.
We put a particular emphasis on the symbolic dynamics behind the expansions and use these dynamics not only as a tool but state some of the main results in symbolic terms.
Since several different problems have the same symbolic background (and are topologically equivalent), these results can then be applied in different settings. 

To state our results, we recall some notation.
The \emph{lexicographic order} is defined by $u_1u_2\cdots \prec v_1v_2\cdots$ if $u_1\cdots u_n = v_1\cdots v_n$ and $u_{n+1} < v_{n+1}$ for some $n \ge 0$ (and $\bu \preceq \bv$ if $\bu \prec \bv$ or $\bu = \bv$). 
For $q_0,q_1>1$ with $q_0{+}q_1 \ge q_0q_1$, the \emph{quasi-greedy} $(q_0,q_1)$-expansion of a number $x \in (0,\frac{1}{q_1-1}]$ is the lexicographically largest sequence $\bu \in \{0,1\}^\infty$ satisfying $\pi_{q_0,q_1}(\bu) = x$ and not ending with~$0^\infty$. 
Similarly, the \emph{quasi-lazy} $(q_0,q_1)$-expansion of a number $x \in [0,\frac{1}{q_1-1})$ is the lexicographically smallest sequence $\bu \in \{0,1\}^\infty$ satisfying $\pi_{q_0,q_1}(\bu) = x$ and not ending with~$1^\infty$. 
Denote the quasi-greedy $(q_0,q_1)$-expansion of $\frac{1}{q_1}$ by~$\ba_{q_0,q_1}$ and the quasi-lazy $(q_0,q_1)$-expansion of $\frac{1}{q_0(q_1-1)}$ by~$\bb_{q_0,q_1}$.
Since $\frac{1}{q_1} = \pi_{q_0,q_1}(10^\infty)$ and $\frac{1}{q_0(q_1-1)} = \pi_{q_0,q_1}(01^\infty)$, the sequence $\ba_{q_0,q_1}$ starts with~$01$ and $\bb_{q_0,q_1}$ starts with~$10$.

By \cite{KomLuZou2022, KomSteZou2022}, the \emph{set of unique $(q_0,q_1)$-expansions} is given by
\[
U_{q_0,q_1} = \big\{i_1i_2\cdots \in \{0,1\}^\infty \,:\, i_ni_{n+1}\cdots \notin [\ba_{q_0,q_1},\bb_{q_0,q_1}]\ \text {for all}\ n \ge 1\big\},
\]
with the closed interval $[\ba,\bb] := \{\bu \in \{0,1\}^\infty : \ba \preceq \bu \preceq \bb\}$. 
It turns out to be more convenient to consider the open interval ${]\ba,\bb[} := \{\bu \in \{0,1\}^\infty : \ba \prec \bu \prec \bb\}$ instead of the closed interval $[\ba,\bb]$, and to define
\[
\Omega_{\ba,\bb} := \big\{i_1i_2\cdots \in \{0,1\}^\infty \,:\, i_ni_{n+1}\cdots \notin {]\ba,\bb[}\ \text {for all}\ n \ge 1\big\}.
\]
Then $\Omega_{\ba,\bb}$ is closed and shift-invariant, i.e., a subshift of $\{0,1\}^\infty$. 
The set $U_{q_0,q_1}$ differs from $\Omega_{\ba_{q_0, q_1},\bb_{q_0, q_1}}$ only by sequences ending in $\ba_{q_0,q_1}$ or $\bb_{q_0,q_1}$, i.e., by a countable set, which implies that
\[
\dim_H\cU_{q_0, q_1} = \dim_H \pi_{q_0, q_1}(\Omega_{\ba_{q_0, q_1},\bb_{q_0, q_1}}).
\]
The sets $\Omega_{\ba, \bb}$ and~$U_{q_0, q_1}$ are closely related to the set of kneading sequences of Lorenz maps and the doubling map with a general hole \cite{HubSpa1990, BarSteVin2014, GleSid2015}

An important role is played by the sequences 
\[
\ell(\ba,\bb) := \max \{\bu \in \Omega_{\ba,\bb} \,:\, \bu \preceq \ba\}, \quad r(\ba,\bb) := \min \{\bu \in \Omega_{\ba,\bb} \,:\, \bu \succeq \bb\},
\]
which are well defined because $\Omega_{\ba,\bb}$ is compact and contains $\{0^\infty,1^\infty\}$.
We have $\ell(\ba,\bb), r(\ba,\bb) \in \Omega_{\ell(\ba,\bb),r(\ba,\bb)} = \Omega_{\ba,\bb}$; see Lemma~\ref{l:laba} below.

Our first main result provides explicit formulae for the Hausdorff dimension of the univoque set~$\cU_{q_0,q_1}$.
Using this theorem, we plot $\dim_H \cU_{q_0,q_1}$ for $1 < q_0,q_1 \le 2.5$ in Figure~\ref{fig:t1}.
The statement of Theorem~\ref{t:t1}~(\ref{i:12}) seems to be new even for $q_0 = q_1 = q$ (but follows rather easily from known results in this case). 

\begin{theorem}\label{t:t1}
Let $s := \dim_H \cU_{q_0, q_1}$, with $q_0,q_1 > 1$.
\begin{enumerate}[\upshape (i)]
\itemsep1ex
\item \label{i:11}
If $q_1 \le \cK(q_0)$, then $s = 0$.
\item \label{i:12}
If $\cK(q_0) < q_1 < \frac{q_0}{q_0-1}$, then $0< s < 1$, and $s$ is the maximal root of 
\[
\pi_{q_0^s,q_1^s}\big(\ell(\ba_{q_0,q_1},\bb_{q_0,q_1})\big) = \pi_{q_0^s,q_1^s}\big(r(\ba_{q_0,q_1},\bb_{q_0,q_1})\big).
\]
In particular, if $\ba_{q_0, q_1}, \bb_{q_0, q_1} \in \Omega_{\ba_{q_0, q_1},\bb_{q_0, q_1}}$, then $s$ is the maximal root of 
\[
\pi_{q_0^s,q_1^s}(\ba_{q_0, q_1}) = \pi_{q_0^s,q_1^s}(\bb_{q_0, q_1}).
\]
\item \label{i:13}
If $q_1 = \frac{q_0}{q_0-1}$, then $s = 1$. 
\item \label{i:14}
If $q_1 > \frac{q_0}{q_0-1}$, then $0 < s < 1$ and $s$ is the maximal root of $q_0^{-s} {+} q_1^{-s} = 1$. 
\end{enumerate}
\end{theorem}

\begin{figure}[ht]
\centerline{\includegraphics{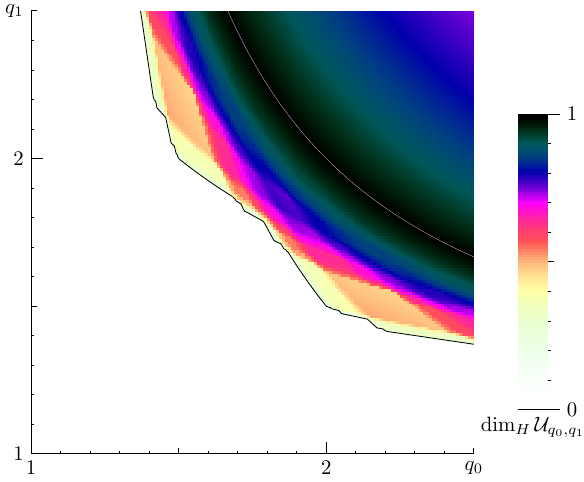}}
\caption{Numerical values of $\dim_H \mathcal{U}_{q_0,q_1}$ for $1 < q_0, q_1 \le 2.5$, represented by colours. The black curve bordering the domain~with $\dim_H \mathcal{U}_{q_0,q_1} = 0$ is $q_1 = \cK(q_0)$, given by \cite[Theorem~2.1]{KomSteZou2022}. We have $\dim_H \mathcal{U}_{q_0,q_1} = 1$ if and only if $(q_0{-}1)(q_1{-}1) = 1$ (thin white curve).} \label{fig:t1}
\end{figure}

The main ingredient of the proof of Theorem~\ref{t:t1} is the following result. 
We state it only for the set of \emph{admissible pairs}
\[
W := \big\{(\ba,\bb) \in 0\{0,1\}^\infty \times 1\{0,1\}^\infty \,:\, \ba,\bb \in \Omega_{\ba,\bb}\big\};
\] 
if $\ba$ or $\bb$ is not in $\Omega_{\ba,\bb}$, then we can replace $(\ba,\bb)$ by $(\ell(\ba,\bb),r(\ba,\bb))$.

\begin{theorem} \label{t:dimH}
Let $(\ba,\bb) \in W$ and $q_0,q_1>1$. 
\begin{enumerate}[\upshape(i)]
\itemsep1ex
\item \label{i:21}
If $\pi_{\beta,\beta}(\ba) < \pi_{\beta,\beta}(\bb)$ for all $\beta > 1$, then
$h(\Omega_{\ba,\bb}) = 0 = \dim_H \pi_{q_0,q_1}(\Omega_{\ba,\bb})$. 
\item \label{i:22}
Otherwise, we have
\[
\begin{aligned}
h(\Omega_{\ba,\bb}) & = \log \beta \ \ \mbox{for the maximal $\beta > 1$ such that}\ \pi_{\beta,\beta}(\ba) = \pi_{\beta,\beta}(\bb), \\
\dim_H \pi_{q_0,q_1}(\Omega_{\ba,\bb}) & = \min\{1,s\} \ \mbox{for the maximal $s > 0$ s.t.\ $\pi_{q_0^s,q_1^s}(\ba) = \pi_{q_0^s,q_1^s}(\bb)$}.
\end{aligned}
\]
\end{enumerate}
\end{theorem}

The following example illustrates these formulas.

\begin{example} \label{ex:1}
Let $q_0 \approx 2.247$ satisfy $q_0^3 = 2 q_0^2 {+} q_0 {-} 1$ and $q_1 = 1 {+} \frac{1}{q_0} \approx 1.445$.
Then the quasi-greedy $(q_0,q_1)$-expansion of $\frac{1}{q_1}$ is $\ba_{q_0,q_1} = 011 (100)^\infty$, the quasi-lazy $(q_0,q_1)$-expansion of $\frac{1}{q_0(q_1-1)} =1 $ is $\bb_{q_0,q_1} = (10)^\infty$, the extremal elements of $\Omega_{\ba_{q_0,q_1},\bb_{q_0,q_1}}$ are $\ell(011(100)^\infty,(10)^\infty) = (011)^\infty$ and $r(011(100)^\infty,(10)^\infty) = (10)^\infty$.
Therefore, the set of unique $(q_0,q_1)$-expansions $U_{q_0,q_1}$ is $\Omega_{(011)^\infty,(10)^\infty}$, except for sequences ending with $(10)^\infty$.
The topological entropy of this set is $h(\Omega_{(011)^\infty,(10)^\infty}) = \log \beta$ with $\beta \approx 1.325$ satisfying $\frac{\beta+1}{\beta^3-1} = \frac{\beta}{\beta^2-1}$, i.e., $\beta^3 = \beta {+} 1$, and the Hausdorff dimension of the univoque set is $\dim_H \cU_{q_0,q_1} = s \approx 0.512$ with $\frac{q_1^s+1}{q_0^sq_1^{2s}-1} = \frac{q_0^s}{q_0^sq_1^s-1}$. 
\end{example}

Moreover, we obtain the following continuity results. 
Let
\[
d_{q_0,q_1}:\, 0\{0,1\}^\infty \times 1\{0,1\}^\infty \to [0,1], \quad (\ba,\bb) \mapsto \dim_H \pi_{q_0,q_1}(\Omega_{\ba,\bb}),
\]
and endow $\{0,1\}^\infty$, as usually, with the product topology of the discrete topology, i.e., two sequences are close when they agree on a long initial part. 
Labarca and Moreira~\cite{LabMor2006} proved that the entropy of the subshift $\Sigma_{\ba,\bb} := \{i_1i_2\cdots \in \{0,1\}^\infty  : i_ni_{n+1}\cdots \in [\ba,\bb]\ \mbox{for all}\ n \ge 1\}$ varies continuously in $(\ba,\bb)$, and we know e.g.\ from \cite[Theorem~2.6]{KomSteZou2022} that $h(\Sigma_{\ba,\bb}) = h(\Omega_{0\bb,1\ba})$.
Our proof of the following theorem is based on Theorem~\ref{t:dimH} and thus rather different from that in~\cite{LabMor2006}. 

\begin{theorem}\label{t:t2a}
Let $q_0,q_1 > 1$. 
Then the function $d_{q_0,q_1}$ is continuous.
\end{theorem}

We can now extend the continuity of $\dim_H \cU_{q,q}$ to that of $\dim_H \cU_{q_0,q_1}$. 
Again, our proof is different from that of \cite{KomKonLi2017,AllKong2019}.

\begin{theorem}\label{t:t2}
The function $(q_0,q_1) \mapsto \dim_H \cU_{q_0,q_1}$ is continuous for $q_0,q_1 > 1$.
\end{theorem}

Finally, let $\bu \in 0\{0,1\}^*$, $\bv \in 1\{0,1\}^*$, where $\{0,1\}^*$ denotes the set of finite words over the alphabet $\{0,1\}$.
If $h(\Omega_{\bu^\infty,\bv^\infty}) > 0$, then we know from \cite[Lemma~8]{BarSteVin2014} that, under certain conditions, $h(\Omega_{\bu\bv^\infty,\bv\bu^\infty}) = h(\Omega_{\bu^\infty,\bv^\infty})$ and thus the entropy function $(\ba,\bb) \mapsto h(\Omega_{\ba,\bb})$ is constant on the rectangle $[\bu^\infty, \bu\bv^\infty]\times [\bv\bu^\infty, \bv^\infty]$; see also \cite[Theorem~3]{SilSR2002}.
The following theorem removes the technical conditions of \cite[Lemma~8]{BarSteVin2014} and shows that a similar result holds for the Hausdorff dimension of $\pi_{q_0, q_1}(\Omega_{\ba, \bb})$.
Note that the condition $h(\Omega_{\bu^\infty,\bv^\infty}) > 0$ cannot be omitted, e.g., we have $\Omega_{0^\infty,1^\infty} = \{0^\infty,1^\infty\}$, $\Omega_{01^\infty,10^\infty} = \{0,1\}^\infty$, and also $h(\Omega_{(01)^\infty,(10)^\infty}) = 0$, $h(\Omega_{01(10)^\infty,10(01)^\infty}) = \frac{1}{2} \log 2$.

\begin{theorem} \label{t:equaldim}
Let $\bu \in 0\{0,1\}^*$, $\bv \in 1\{0,1\}^*$.
If $h(\Omega_{\bu^\infty,\bv^\infty}) > 0$, then
\begin{equation} \label{e:equaldimuv} d_{q_0,q_1}(\bu^\infty,\bv^\infty) = d_{q_0,q_1}(\bu\bv^\infty,\bv\bu^\infty) \quad \mbox{for all}\ q_0,q_1 > 1,
\end{equation}
thus $h(\Omega_{\bu^\infty,\bv^\infty}) = h(\Omega_{\bu\bv^\infty,\bv\bu^\infty})$.   

For all $\bu \in 0\{0,1\}^*$\!, $\bv \in 1\{0,1\}^*$\!, $\ba \in 0\{0,1\}^\infty$\!, $\bb \in 1\{0,1\}^\infty$\!, $q_0,q_1 > 1$, we~have
\[
d_{q_0,q_1}(\bu^\infty,\bb) = d_{q_0,q_1}(\bu\bb,\bb) \quad \mbox{and} \quad d_{q_0,q_1}(\ba,\bv^\infty) = d_{q_0,q_1}(\ba,\bv\ba),
\]
thus $h(\Omega_{\bu^\infty,\bb}) = h(\Omega_{\bu\bb,\bb})$ and $h(\Omega_{\ba,\bv^\infty}) = h(\Omega_{\ba,\bv\ba})$.
\end{theorem}

We prove Theorems~\ref{t:t1} and~\ref{t:dimH} in Section~\ref{sec:calc-hausd-dimens}, Theorems~\ref{t:t2a}--\ref{t:equaldim} in Section~\ref{sec:hausd-dimens-cu_q_0}.

\section{Calculating the Hausdorff dimension of $\pi_{q_0,q_1}(\Omega_{\ba,\bb})$} \label{sec:calc-hausd-dimens}
We first give a description of the Hausdorff dimension of $\pi_{q_0,q_1}(\Omega_{\ba,\bb})$ by using the language of~$\Omega_{\ba,\bb}$, similarly to the entropy. 
This will play an important role in the proof of Theorem~\ref{t:dimH}.
Let the cylinder (in~$\Omega_{\ba,\bb}$) of a word $i_1\cdots i_n \in \{0,1\}^*$ be
\[
[i_1\cdots i_n]_{\ba,\bb} := \Omega_{\ba,\bb} \cap i_1 \cdots i_n \{0,1\}^\infty,
\]
and set
\begin{equation} \label{e:Labn}
L_{\ba,\bb,n} := \{i_1\cdots i_n : i_1i_2 \cdots \in \Omega_{\ba,\bb}\} = \{i_1\cdots i_n \in \{0,1\}^n : [i_1\cdots i_n]_{\ba,\bb} \ne \emptyset\}.\hspace{-.5em}
\end{equation}

\begin{lemma} \label{l:Barreira}
Let $\ba \in 0\{0,1\}^\infty$, $\bb \in 1\{0,1\}^\infty$, $q_0,q_1>1$, with
\begin{equation} \label{e:maxmin}
\max \pi_{q_0,q_1}\big([0]_{\ba,\bb}\big) < \min \pi_{q_0,q_1}\big([1]_{\ba,\bb}\big).
\end{equation}
Then $d_{q_0,q_1}(\ba,\bb) = s$, where $s$ is the unique root of the equation
\begin{equation} \label{e:hausdoffD2}
\lim_{n\to\infty} \frac{1}{n} \log \sum_{i_1\cdots i_n\in L_{\ba,\bb,n}} \frac{1}{q_{i_1}^s\cdots q_{i_n}^s} = 0.
\end{equation}
In particular, $d_{q_0,q_1}(\ba,\bb) = 0$ if and only if $h(\Omega_{\ba,\bb}) = 0$. 
\end{lemma}

\begin{proof}
Assume that $\ba,\bb,q_0,q_1$ satisfy \eqref{e:maxmin}.
For $i_1 \cdots i_n \in L_{\ba,\bb,n}$, $n \ge 1$, let
\[
\Delta_{\ba,\bb,i_1\cdots i_n} := \bigg[\min \pi_{q_0,q_1}\big([i_1\cdots i_n]_{\ba,\bb}\big),\, \max \pi_{q_0,q_1}\big([i_1\cdots i_n]_{\ba,\bb}\big) + \frac{C}{q_{i_1}\cdots q_{i_{n-1}}}\bigg],
\]
with
\[
C := \min \pi_{q_0,q_1}\big([1]_{\ba,\bb}\big) - \max  \pi_{q_0,q_1}\big([0]_{\ba,\bb}\big).
\]
Then 
\[
\pi_{q_0,q_1}(\Omega_{\ba,\bb}) = \bigcap_{n=1}^\infty \bigcup_{i_1\cdots i_n \in L_{\ba,\bb,n}} \Delta_{\ba,\bb,i_1\cdots i_n}
\]
is a \emph{generalised Moran construction} in the sense of \cite[Section~2.1]{Barreira1996}.
Indeed, we~have
\[
\frac{\min\{q_0,q_1\}\,C}{q_{i_1}\cdots q_{i_n}} \le \mathrm{diam}(\Delta_{\ba,\bb,i_1\cdots i_n}) \le \frac{1/(q_1{-}1){+}\max\{q_0,q_1\}\,C}{q_{i_1}\cdots q_{i_n}}
\]
for all $i_1\cdots i_n \in L_{\ba,\bb,n}$.  The interiors of $\Delta_{\ba,\bb,i_1\cdots i_n}$ and $\Delta_{\ba,\bb,j_1\cdots j_n}$ are disjoint for all distinct $i_1 \cdots i_n, j_1 \cdots j_n \in L_{\ba,\bb,n}$, $n \ge 1$, because we can assume w.l.o.g.\ that $i_1 \cdots i_k = j_1 \cdots j_k$, $i_{k+1} = 0$, $j_{k+1} = 1$ for some $0 \le k < n$, thus
\[
\begin{aligned}
\max \pi_{q_0,q_1}\big([i_1\cdots i_n]_{\ba,\bb}\big) {+} \frac{C}{q_{i_1}\cdots q_{i_{n-1}}} & \le \max \pi_{q_0,q_1}\big([i_1\cdots i_k0]_{\ba,\bb}\big) {+} \frac{C}{q_{i_1}\cdots q_{i_k}} \\
& \hspace{-7em} \le \min \pi_{q_0,q_1}\big([i_1\cdots i_k1]_{\ba,\bb}\big) \le \min \pi_{q_0,q_1}\big([j_1\cdots j_n]_{\ba,\bb}\big).
\end{aligned}
\]
Since $\log (q_{i_1} \cdots q_{i_n})^{-1}$ is an additive function in~$n$, $(q_{i_1} \cdots q_{i_n})^{-1} \le (\min\{q_0,q_1\})^{-n}$ and $(q_{i_1} \cdots q_{i_{n+1}})^{-1} \ge (\max\{q_0,q_1\})^{-1} (q_{i_1} \cdots q_{i_n})^{-1}$ for all sequences $i_1\cdots i_{n+1} \in \{0,1\}^{n+1}$, it follows from \cite[Theorem~2.1]{Barreira1996} that $d_{q_0,q_1}(\ba,\bb)$ equals the unique root~$s$ of the equation~\eqref{e:hausdoffD2}.
Since $h(\Omega_{\ba,\bb}) = \lim_{n\to\infty} \frac{1}{n} \log \# L_{\ba,\bb,n}$, we have $h(\Omega_{\ba,\bb}) = 0$ if and only if \eqref{e:hausdoffD2} holds for $s = 0$.
\end{proof}

Next, we relate equation \eqref{e:hausdoffD2} to
\begin{equation} \label{e:abs}
\pi_{q_0^s,q_1^s}(\ba) = \pi_{q_0^s,q_1^s}(\bb).
\end{equation}

\begin{lemma} \label{l:K}
Let $(\ba,\bb) \in W$, $q_0,q_1>1$.
If \eqref{e:hausdoffD2} holds for $s > 0$, then $s$ is the maximal root of \eqref{e:abs}.
We have $h(\Omega_{\ba,\bb}) = 0$, i.e., \eqref{e:hausdoffD2} holds for $s = 0$, if and only if \eqref{e:abs} has no root $s > 0$.
\end{lemma}

For the proof, we use the sets
\[
Q_{\ba,\bb,n} := \big\{i_1\cdots i_n \in L_{\ba,\bb,n} \,:\, i_1\cdots i_n\mathbf{a},\, i_1\cdots i_n\mathbf{b} \in \Omega_{\ba, \bb}\big\}
\]
(with $Q_{\ba,\bb,0} = L_{\ba,\bb,0}$ being the set containing only the empty word), in addition to the sets $L_{\ba,\bb,n}$ from~\eqref{e:Labn}.
We recall the following decomposition of words from~$L_{\ba,\bb,n}$ into words from~$Q_{\ba,\bb,k}$ and prefixes of~$\ba$ and~$\bb$ respectively that was already used e.g.\ in \cite[Proof of Lemma~3]{BarSteVin2014}.

\begin{lemma} \label{l:decomposition}
Let $(\ba,\bb) = (a_1a_2\cdots, b_1b_2\cdots) \in W$.
Then we have the following.
\begin{enumerate}[\upshape (i)]
\item \label{i:Qa}
For each $i_1 \cdots i_n \in L_{\ba,\bb,n} \setminus \{1^n\} $, $n \ge 1$, there is a unique $0 \le k < n$ such that
\begin{equation}
\label{e:decomposition1}
i_1 \cdots i_k \in Q_{\ba,\bb,k} \quad \mbox{and} \quad  i_{k+1} \cdots i_n = a_1 \cdots a_{n-k}.
\end{equation}
\item \label{i:Qb}
For each $i_1 \cdots i_n \in L_{\ba,\bb,n} \setminus \{0^n\} $, $n \ge 1$, there is a unique $0 \le k < n$ such that
\[
i_1 \cdots i_k \in Q_{\ba,\bb,k} \quad \mbox{and} \quad  i_{k+1} \cdots i_n = b_1 \cdots b_{n-k}.
\]
\end{enumerate}
\end{lemma}

\begin{proof}
It suffices to prove (\ref{i:Qa}) because (\ref{i:Qb}) is symmetric. 

Let $i_1 \cdots i_n \in L_{\ba,\bb,n}$. 
First note that $i_{k+1} \cdots i_n = a_1 \cdots a_{n-k}$ for some $0 \le k < n$ if and only if $i_1 \cdots i_n \ne 1^n$ because in all other cases there exists $k < n$ such that $i_{k+1} \cdots i_n = 01^{n-k-1}$, then $i_1 \cdots i_n \in L_{\ba,\bb,n}$ implies that $01^{n-k-1} \in L_{\ba,\bb,n-k}$ and thus $01^{n-k-1}  = a_1 \cdots a_{n-k}$.

Next we show that $i_1 \cdots i_k \in Q_{\ba,\bb,k}$ for the minimal $k < n$ such that $i_{k+1} \cdots i_n = a_1 \cdots a_{n-k}$.
Since $\ba, \bb \in \Omega_{\ba,\bb}$, we only have to prove for all $j < k$ that $i_{j+1} \cdots i_k \bb \preceq \ba$ when $i_{j+1} = 0$ and $i_{j+1} \cdots i_k \ba \succeq \bb$ when $i_{j+1} = 1$. 
Assume first that $i_{j+1} = 0$. 
Then $i_1 \cdots i_n \in L_{\ba,\bb,n}$ implies that $i_{j+1}\cdots i_n \preceq a_1 \cdots a_{n-j}$ (where the lexicographic order is defined for finite words of same length in the same way as for infinite words). 
If $i_{j+1} \cdots i_k \prec a_1 \cdots a_{k-j}$, then $i_{j+1} \cdots i_k \bb \prec \ba$.
If $i_{j+1} \cdots i_k = a_1 \cdots a_{k-j}$, then $a_{k-j+1} \cdots a_{n-j} \succeq i_{k+1} \cdots i_n = a_1 \cdots a_{n-k}$, thus $a_{k-j+1} = 0$ would imply that $a_{k-j+1} \cdots a_{n-j} = a_1 \cdots a_{n-k}$ and hence $i_{j+1} \cdots i_n = a_1 \cdots a_{n-j}$, contradicting that $k$ is minimal, while $a_{k-j+1} = 1$ implies that $\ba \succeq a_1 \cdots a_{k-j} \bb =  i_{j+1} \cdots i_k \bb$.
Assume now that $i_{j+1} {\,=\,} 1$.
Then $i_{j+1} \cdots i_k {\,\succeq\,} b_1 \cdots b_{k-j}$, thus $i_{j+1} \cdots i_k \ba \succ \bb$, or $i_{j+1} \cdots i_k = b_1 \cdots b_{k-j}$, thus $b_{k-j+1} \cdots b_{n-j} \preceq i_{k+1} \cdots i_n = a_1 \cdots a_{n-k}$, and $\bb \preceq b_1 \cdots b_{k-j} \ba = i_{j+1} \cdots i_k \ba$. 
This proves that $i_1 \cdots i_k \in Q_{\ba,\bb,k}$.

Consider now $k < j < n$ such that $i_{j+1} \cdots i_n = a_1 \cdots a_{n-j}$. 
Then $i_{k+1} \cdots i_j \bb = a_1 \cdots a_{j-k} \bb \succ a_1 \cdots a_{j-k} \ba \succeq \ba$, thus $i_{k+1} \cdots i_j \bb \in (\ba,\bb)$ and hence $i_1 \cdots i_j \notin Q_{\ba,\bb,j}$.
This proves that there is a unique $k < n$ satisfying~\eqref{e:decomposition1}.
\end{proof}

In terms of the formal power series
\[
L_{\ba,\bb}(z_0,z_1) := \sum_{n=1}^\infty \sum_{i_1\cdots i_n \in L_{\ba,\bb,n}} \hspace{-1em} z_{i_1} \cdots z_{i_ n}, \quad Q_{\ba,\bb}(z_0,z_1) := \sum_{n=0}^\infty \sum_{i_1\cdots i_n \in Q_{\ba,\bb,n}} \hspace{-1em} z_{i_1} \cdots z_{i_n},
\]
Lemma~\ref{l:decomposition}~(\ref{i:Qa}) means that
\[
\sum_{i_1\cdots i_n \in L_{\ba,\bb,n}} z_{i_1} \cdots z_{i_ n} = z_1^n + \sum_{k=0}^{n-1} \sum_{i_1\cdots i_k \in Q_{\ba,\bb,k}} z_{i_1} \cdots z_{i_k}\, z_{a_1} \cdots z_{a_{n-k}}
\]
for all $n \ge 1$.
Setting $A_{\ba,\bb}(z_0,z_1) := \sum_{n=1}^\infty z_{a_1}\cdots z_{a_n}$, we obtain that
\begin{equation} \label{e:LAQ}
L_{\ba,\bb}(z_0,z_1) = \frac{z_1}{1{-}z_1} + A_{\ba,\bb}(z_0,z_1) Q_{\ba,\bb}(z_0,z_1).
\end{equation}
The \emph{kneading invariant}
\begin{equation*} \label{e:K}
K_{\ba,\bb}(z_0,z_1) := \sum_{n=1}^\infty \big(b_n z_{b_1} \cdots z_{b_n} -  a_n z_{a_1} \cdots z_{a_n}\big)
\end{equation*}
is important for us because
\[
K_{\ba,\bb}(q_0^{-1}\!,q_1^{-1}) = \pi_{q_0,q_1}(\bb) - \pi_{q_0,q_1}(\ba)
\]
for all $q_0,q_1 > 1$.
Using Lemma~\ref{l:decomposition}~(\ref{i:Qa}) and~(\ref{i:Qb}), we obtain that
\[
\begin{aligned}
\sum_{i_1\cdots i_n \in L_{\ba,\bb,n}} i_n z_{i_1} \cdots z_{i_ n} & = z_1^n + \sum_{k=0}^{n-1} \sum_{i_1\cdots i_k \in Q_{\ba,\bb,k}} z_{i_1} \cdots z_{i_k}\, a_{n-k} z_{a_1} \cdots z_{a_{n-k}}, \\
\sum_{i_1\cdots i_n \in L_{\ba,\bb,n}} i_n z_{i_1} \cdots z_{i_ n} & = \sum_{k=0}^{n-1} \sum_{i_1\cdots i_k \in Q_{\ba,\bb,k}} z_{i_1} \cdots z_{i_k}\, b_{n-k} z_{b_1} \cdots z_{b_{n-k}},
\end{aligned}
\]
thus
\begin{equation} \label{e:QK}
Q_{\ba,\bb}(z_0,z_1) K_{\ba,\bb}(z_0,z_1) = \frac{z_1}{1-z_1}.
\end{equation}

\begin{proof}[Proof of Lemma~\ref{l:K}]
Let $s \ge 0$ be the unique root of \eqref{e:hausdoffD2}.
For every $t > s$, we have $0 < L_{\ba,\bb}(q_0^{-t},q_1^{-t}) < \infty$, and the equations \eqref{e:LAQ} and~\eqref{e:QK} hold for $(z_0,z_1) = (q_0^{-t},q_1^{-t})$ because the series converge absolutely, thus $K_{\ba,\bb}(q_0^{-t},q_1^{-t}) > 0$.
In particular, $s = 0$ implies that \eqref{e:abs} has no root $s > 0$.
Since the subadditivity of $\log \sum_{i_1\cdots i_n\in L_{\ba,\bb,n}} q_{i_1}^{-s} \cdots q_{i_n}^{-s}$ implies that 
\[
\inf_{n\ge1} \frac{1}{n} \log \sum_{i_1\cdots i_n\in L_{\ba,\bb,n}} \frac{1}{q_{i_1}^s\cdots q_{i_n}^s} = \lim_{n\to\infty} \frac{1}{n} \log \sum_{i_1\cdots i_n\in L_{\ba,\bb,n}} \frac{1}{q_{i_1}^s\cdots q_{i_n}^s} = 0,
\]
we have $\sum_{i_1\cdots i_n\in L_{\ba,\bb,n}} q_{i_1}^{-s} \cdots q_{i_n}^{-s} \ge 1$ for all $n \ge 1$, thus $\lim_{t\downarrow s} L_{\ba,\bb}(q_0^{-t},q_1^{-t}) = \infty$.
When $s > 0$, then \eqref{e:LAQ} and~\eqref{e:QK} give that $\lim_{t\downarrow s} K_{\ba,\bb}(q_0^{-t},q_1^{-t}) = 0$ and, by continuity, $K_{\ba,\bb}(q_0^{-s},q_1^{-s}) = 0$, thus $s$ is the maximal root of~\eqref{e:abs}.

Assume now that \eqref{e:abs} has a root $s > 0$, and assume that $s$ is maximal. 
Since $K_{\ba,\bb}(q_0^{-t},q_1^{-t}) > 0$ for large~$t$, we obtain similarly to the previous paragraph that $s$ is a root of~\eqref{e:hausdoffD2}. Then \eqref{e:hausdoffD2} does not hold for $s = 0$, i.e., $h(\Omega_{\ba,\bb}) > 0$.
\end{proof}

Next, we consider the case that~\eqref{e:maxmin} does not hold.

\begin{lemma} \label{l:full}
Let $\ba \in 0\{0,1\}^\infty$, $\bb \in 1\{0,1\}^\infty$, $q_0, q_1 > 1$.
If $\max \pi_{q_0,q_1}([0]_{\ba,\bb}) \ge \min \pi_{q_0,q_1}([1]_{\ba,\bb})$, then $\pi_{q_0,q_1}(\Omega_{\ba,\bb}) = \big[0,\frac{1}{q_1-1}\big]$. 
\end{lemma}

In the proof, we use the maps $T_0$ and~$T_1$ (that depend on $q_0,q_1$), with 
\begin{equation} \label{e:Ti}
T_i:\, \mathbb{R} \to \mathbb{R}, \quad x \mapsto q_i x - i;
\end{equation}
see Figure~\ref{f:T01}.

\begin{figure}[ht]
\centerline{\begin{tikzpicture}[scale=1.5]
\draw(0,0)node[below]{$0$}node[left]{$0$}--(2.247,0)node[below]{$1/(q_1{-}1)$}--(2.247,2.247)--(0,2.247)node[left]{$1/(q_1{-}1)$}--cycle;
\draw[thick](0,0)--node[above,rotate=66,pos=.6]{$T_0(x) = q_0 x$}(1,2.247) (.692,0)--node[above,rotate=55]{$T_1(x) = q_1 x{-}1$}(2.247,2.247);
\end{tikzpicture}}
\caption{The maps $T_0, T_1$ restricted to the interval $[0,\frac{1}{q_1-1}]$, with $q_0,q_1$ as in Example~\ref{ex:1}.} \label{f:T01}
\end{figure}

\begin{proof}
Let $\ba' = a'_1a'_2\cdots \in [0]_{\ba,\bb}$, $\bb' = b'_1b'_2\cdots \in [1]_{\ba,\bb}$ be such that $\pi_{q_0,q_1}(\ba') = \max \pi_{q_0,q_1}([0]_{\ba,\bb}) \ge \min \pi_{q_0,q_1}([1]_{\ba,\bb}) = \pi_{q_0,q_1}(\bb')$.

For $x \in \big[0,\frac{1}{q_1-1}\big]$, define a sequence $i_1i_2\cdots \in \{0,1\}^\infty$ by setting
\[
i_n = \begin{cases}1 & \mbox{if}\ T_{i_{n-1}} \circ \cdots \circ T_{i_1}(x) > \pi_{q_0,q_1}(\ba'), \\ 0 & \mbox{otherwise,} \end{cases}
\]
recursively for $n \ge 1$.
Since $T_0\big([0,\pi_{q_0,q_1}(\ba')]\big) = [0,\pi_{q_0,q_1}(a'_2a'_3\cdots)] \subseteq \big[0,\tfrac{1}{q_1-1}\big]$ and 
\[
T_1\big(\big[\pi_{q_0,q_1}(\ba'),\tfrac{1}{q_1-1}\big]\big) {\,\subseteq\,} T_1\big(\big[\pi_{q_0,q_1}(\bb'),\tfrac{1}{q_1-1}\big]\big) {\,=\,} \big[\pi_{q_0,q_1}(b'_2b'_3\cdots),\tfrac{1}{q_1-1}\big] {\,\subseteq\,} \big[0,\tfrac{1}{q_1-1}\big],
\]
we have $T_{i_n} \circ \cdots \circ T_{i_1}(x) \in \big[0,\frac{1}{q_1-1}\big]$ for all $n \ge 1$, thus $\pi_{q_0,q_1}(i_1i_2\cdots) = x$.

To show that $i_1i_2\cdots \in \Omega_{\ba,\bb}$, we prove, by induction on $k \ge 0$, that $i_n \cdots i_{n+k} \preceq a'_1 \cdots a'_{k+1}$ for all $n \ge 1$ such that $i_n = 0$, and $i_n \cdots i_{n+k} \succeq b'_1 \cdots b'_{k+1}$ for all $n \ge 1$ such that $i_n = 1$.
This is clearly true for $k = 0$.
Assume that $i_n \cdots i_{n+{k-1}} \preceq a'_1 \cdots a'_k$, with $k,n \ge 1$.
If $i_n \cdots i_{n+{k-1}} \prec a'_1 \cdots a'_k$ or $i'_{n+k} = 0$, then we also have $i_n \cdots i_{n+k} \preceq a'_1 \cdots a'_{k+1}$.
If $i_n \cdots i_{n+k-1} = a'_1 \cdots a'_k$ and $i_{n+k} = 1$, then
\[
\begin{aligned}
& \pi_{q_0,q_1}(a'_{k+1}a'_{k+2}\cdots) = T_{a'_k} \circ \cdots \circ T_{a'_1}(\pi_{q_0,q_1}(\ba')) \\
& \quad \ge T_{a'_k} \circ \cdots \circ T_{a'_1} \circ T_{i_{n-1}} \circ \cdots \circ T_{i_1}(x) = T_{i_{n+k-1}} \circ \cdots \circ T_{i_1}(x) > \pi_{q_0,q_1}(\ba'),
\end{aligned}
\]
where we have used that $i_n = 0$, $T_0$ and~$T_1$ are monotonically increasing, and that $i_{n+k} = 1$.
Therefore, we have $a'_{k+1}a'_{k+2}\cdots \notin [0]_{\ba,\bb}$. 
Since $\ba' \in \Omega_{\ba,\bb}$ implies that $a'_{k+1}a'_{k+2}\cdots \in \Omega_{\ba,\bb}$, we obtain that $a'_{k+1} = 1$, hence $i_n \cdots i_{n+k} = a'_1 \cdots a'_{k+1}$.
Assume now $i_n \cdots i_{n+{k-1}} \succeq b'_1 \cdots b'_k$. 
If $i_n \cdots i_{n+{k-1}} \succ b'_1 \cdots b'_k$ or $i_{n+k} = 1$, then we also have $i_n \cdots i_{n+k} \succeq b'_1 \cdots b'_{k+1}$.
If $i_n \cdots i_{n+k-1} = b'_1 \cdots b'_k$ and $i_{n+k} = 0$, then
\[
\begin{aligned}
& \pi_{q_0,q_1}(b'_{k+1}b'_{k+2}\cdots) = T_{b'_k} \circ \cdots \circ T_{b'_1}(\pi_{q_0,q_1}(\bb')) \\
& = T_{i_{n+k-1}} {\circ} \cdots {\circ} T_{i_1}(x) - \big(T_{b'_k} {\circ} \cdots {\circ} T_{b'_1} {\circ} T_{i_{n-1}} {\circ} \cdots {\circ} T_{i_1}(x) {-} T_{b'_k} {\circ} \cdots {\circ} T_{b'_1}(\pi_{q_0,q_1}(\bb'))\big) \\
& = T_{i_{n+k-1}} \circ \cdots \circ T_{i_1}(x) - q_{b'_k} \cdots q_{b'_1} \big(T_{i_{n-1}} \circ \cdots \circ T_{i_1}(x) - \pi_{q_0,q_1}(\bb')\big) \\ 
& < \pi_{q_0,q_1}(\ba') - q_{b'_k} \cdots q_{b'_1} \big(\pi_{q_0,q_1}(\ba') - \pi_{q_0,q_1}(\bb')\big) \\
& \le \pi_{q_0,q_1}(\ba') - \big(\pi_{q_0,q_1}(\ba') - \pi_{q_0,q_1}(\bb')\big) = \pi_{q_0,q_1}(\bb'),
\end{aligned}
\]
where we have used that $i_n = 1$, $i_{n+k} = 1$, and that $\pi_{q_0,q_1}(\ba') \ge \pi_{q_0,q_1}(\bb')$.
This implies that $b'_{k+1} = 0$, hence $i_n \cdots i_{n+k} = b'_1 \cdots b'_{k+1}$.
We have shown that $i_ni_{n+1}\cdots \preceq \ba' \preceq \ba$ for all $n \ge 1$ such that $i_n = 0$ and $i_ni_{n+1}\cdots \succeq \bb' \succeq \bb$ for all $n \ge 1$ such that $i_n = 1$, thus $i_ni_{n+1}\cdots \in \Omega_{\ba',\bb'} \subseteq \Omega_{\ba,\bb}$. 
\end{proof}

We obtain the following characterisation of the Hausdorff dimension of~$\Omega_{\ba,\bb}$. 

\begin{proposition} \label{p:K}
Let $(\ba,\bb) \in W$, $q_0,q_1>1$.
Then $d_{q_0,q_1}(\ba,\bb) = \min\{1,s\}$, where $s$ is the maximal root of the equation $\pi_{q_0^s,q_1^s}(\ba) = \pi_{q_0^s,q_1^s}(\bb)$ if it has a root $s > 0$, $s = 0$ if it has no such root.
We have $d_{q_0,q_1}(\ba,\bb) = 0$ if and only if $h(\Omega_{\ba,\bb}) = 0$, $d_{q_0,q_1}(\ba,\bb) = 1$ if and only if $\pi_{q_0,q_1}(\Omega_{\ba,\bb}) = \big[0,\frac{1}{q_1-1}\big]$.
\end{proposition}

\begin{proof}
  If $\max \pi_{q_0,q_1}([0]_{\ba,\bb}) < \min \pi_{q_0,q_1}([1]_{\ba,\bb})$, then the formula for $d_{q_0,q_1}(\ba,\bb)$ and the equivalence of $d_{q_0,q_1}(\ba,\bb) = 0$ and $h(\Omega_{\ba,\bb}) = 0$ follow from Lemmas~\ref{l:Barreira} and~\ref{l:K}; note that $s < 1$ in this case because $\pi_{q_0,q_1}(\ba) < \pi_{q_0,q_1}(\bb)$ and the Hausdorff dimension cannot exceed~$1$.
We have $\pi_{q_0,q_1}(\Omega_{\ba,\bb}) \ne \big[0,\frac{1}{q_1-1}\big]$ because $x \notin \pi_{q_0,q_1}(\Omega_{\ba,\bb})$ for all $\max \pi_{q_0,q_1}([0]_{\ba,\bb}) < x < \min \pi_{q_0,q_1}([1]_{\ba,\bb})$.

If $\max \pi_{q_0,q_1}([0]_{\ba,\bb}) \ge \min \pi_{q_0,q_1}([1]_{\ba,\bb})$, then $\pi_{q_0,q_1}(\Omega_{\ba,\bb}) = \big[0,\frac{1}{q_1-1}\big]$ and thus $d_{q_0,q_1}(\ba,\bb) = 1$ by Lemma~\ref{l:full}.
Since $\max \pi_{q_0^t,q_1^t}([0]_{\ba,\bb}) < \min \pi_{q_0^t,q_1^t}([1]_{\ba,\bb})$ for sufficiently large~$t$, we have, by continuity, $\max \pi_{q_0^s,q_1^s}([0]_{\ba,\bb}) = \min \pi_{q_0^s,q_1^s}([1]_{\ba,\bb})$ for some $s \ge 1$.
Since $\ba$ is the lexicographically largest element of $[0]_{\ba,\bb}$, this implies that $\pi_{q_0^s,q_1^s}(\ba) = \max \pi_{q_0^s,q_1^s}([0]_{\ba,\bb})$ and, symmetrically, $\pi_{q_0^s,q_1^s}(\bb) = \min \pi_{q_0^s,q_1^s}([1]_{\ba,\bb})$, thus $\pi_{q_0^s,q_1^s}(\ba) = \pi_{q_0^s,q_1^s}(\bb)$.
By Lemma~\ref{l:K}, we have $h(\Omega_{\ba,\bb}) > 0$ in this case.
\end{proof}

We recall (and prove) the following result for the entropy.

\begin{proposition} \label{p:entropy}
Let $(\ba,\bb) \in W$.
Then $h(\Omega_{\ba,\bb}) = \log \beta$, where $\beta$ is the maximal solution of the equation $\pi_{\beta,\beta}(\ba) = \pi_{\beta,\beta}(\bb)$ if a solution $\beta > 1$ exists, $\beta = 1$ if no such solution exists. 
\end{proposition}

\begin{proof}
This follows from $h(\Omega_{\ba,\bb}) = \lim_{n\to\infty} \frac{1}{n} \log \# L_{\ba,\bb,n}$ and Lemma~\ref{l:K}.
Indeed, if $h(\Omega_{\ba,\bb}) = \log \beta$ with $\beta > 1$, then \eqref{e:hausdoffD2} holds for $q_0 = q_1 = \beta$, $s = 1$, thus $s = 1$ is the maximal root of the equation $\pi_{\beta^s,\beta^s}(\ba) = \pi_{\beta^s,\beta^s}(\bb)$ by Lemma~\ref{l:K}.
If $h(\Omega_{\ba,\bb}) = 0$, then \eqref{e:hausdoffD2} holds for $q_0 = q_1 = 2$, $s = 0$, and $\pi_{2^s,2^s}(\ba) = \pi_{2^s,2^s}(\bb)$ has no solution $s > 0$ by Lemma~\ref{l:K}.
\end{proof}

We have proved Theorem~\ref{t:dimH}.

\begin{proof}[Proof of Theorem~\ref{t:dimH}]
This follows from Propositions~\ref{p:K} and~\ref{p:entropy}.
Indeed, the inequality $\pi_{q,q}(\ba) < \pi_{q,q}(\bb)$ for all $q > 1$ is equivalent to $\pi_{q,q}(\ba) \ne \pi_{q,q}(\bb)$ for all $q > 1$ because $\pi_{q,q}(\ba) < \pi_{q,q}(\bb)$ for sufficiently large~$q$, and, for all $q_0,q_1>1$, the solution of \eqref{e:hausdoffD2} is $s = 0$ if and only if $h(\Omega_{\ba,\bb}) = 0$.
\end{proof}

The proof of Theorem~\ref{t:t1} follows from Theorem~\ref{t:dimH} and the following lemma.

\begin{lemma} \label{l:laba}
Let $\ba,\bb \in \{0,1\}^\infty$ with $\ba \preceq \bb$. 
Then 
\[
\Omega_{\ell(\ba,\bb),r(\ba,\bb)} = \Omega_{\ba,\bb},
\]
in particular $\ell(\ba,\bb), r(\ba,\bb) \in \Omega_{\ell(\ba,\bb),r(\ba,\bb)}$.
\end{lemma}

\begin{proof}
Since ${]}\ba,\bb{[} \subseteq {]}\ell(\ba,\bb),r(\ba,\bb){[}$, we have $\Omega_{\ell(\ba,\bb),r(\ba,\bb)} \subseteq \Omega_{\ba,\bb}$.
To see that $\Omega_{\ba,\bb} \subseteq \Omega_{\ell(\ba,\bb),r(\ba,\bb)}$, suppose that there exists $i_1i_2\cdots \in \Omega_{\ba,\bb} \setminus \Omega_{\ell(\ba,\bb),r(\ba,\bb)}$.
This would imply that $i_ni_{n+1}\cdots \in {]}\ell(\ba,\bb),\ba] \cup [\bb,r(\ba,\bb){[}$ for some $n \ge 1$ and $i_ni_{n+1}\cdots \in \Omega_{\ba,\bb}$, contradicting the definition of $\ell(\ba,\bb),r(\ba,\bb)$. 
\end{proof}

\begin{proof}[Proof of Theorem~\ref{t:t1}]
If $q_1 \le \cK(q_0)$, then $h(U_{q_0,q_1}) = 0$ and thus $\dim_H \cU_{q_0, q_1} = 0$.

If $q_1 \ge \frac{q_0}{q_0-1}$, then $U_{q_0,q_1} = \{0,1\}^\infty = \Omega_{01^\infty,10^\infty}$ (up to countably many sequences when $q_1 = \frac{q_0}{q_0-1}$), thus $\dim_H \cU_{q_0, q_1} = s$ with $q_0^{-s} + q_1^{-s} = 1$.

Let now $\cK(q_0) < q_1 < \frac{q_0}{q_0-1}$.
Let $\ba' = \ell(\ba_{q_0,q_1},\bb_{q_0,q_1})$ and $\bb' = r(\ba_{q_0,q_1},\bb_{q_0,q_1})$.
Then $\ba', \bb'  \in \Omega_{\ba',\bb'} = \Omega_{\ba_{q_0, q_1},\bb_{q_0, q_1}}$ by Lemma~\ref{l:laba}.
Moreover, we have 
\[
\begin{gathered}
\pi_{q_0,q_1}(\ba') \le \pi_{q_0,q_1}(\ba_{q_0, q_1}) = \tfrac{1}{q_1} < \tfrac{1}{q_0(q_1-1)} = \pi_{q_0,q_1}(\bb_{q_0, q_1}) \le \pi_{q_0,q_1}(\bb'), \\
\pi_{q_0,q_1}(\ba') = \max \pi_{q_0,q_1}([0]_{\ba',\bb'}), \quad and \quad \pi_{q_0,q_1}(\bb') = \min \pi_{q_0,q_1}([1]_{\ba',\bb'}).
\end{gathered}
\]
Since $h(\Omega_{\ba',\bb'}) > 0$, the equation $\pi_{q_0^s,q_1^s}(\ba') = \pi_{q_0^s,q_1^s}(\bb')$ has a maximal positive root $0 < s < 1$.
By Theorem~\ref{t:dimH}, we have $\dim_H(\cU_{q_0, q_1}) = s$.
\end{proof}

\section{Continuity of the Hausdorff dimension of $\pi_{q_0,q_1}(\Omega_{\ba,\bb})$} \label{sec:hausd-dimens-cu_q_0}
To prove that $d_{q_0,q_1}$ is continuous, we show first that we can restrict to~$W$.

\begin{lemma} \label{l:abinOmega}
If $d_{q_0,q_1}$ is continuous at all points $(\ba,\bb) \in W$, then it is continuous.
\end{lemma}

For the proof, we determine $\ell(\ba,\bb)$ and $r(\ba,\bb)$.

\begin{lemma} \label{l:ab}
Let $\ba = a_1a_2\cdots \preceq\bb = b_1b_2\cdots$.
Then we have the following. 
\begin{enumerate}[\upshape (i)]
\itemsep.5ex
\item \label{i:ab1}
 If $a_{n+1}a_{n+2}\cdots \in{]\ba,\bb[}$ for some $n \ge 1$, then $\Omega_{\ba,\bb} = \Omega_{(a_1\cdots a_n)^\infty,\bb}$. 
\item \label{i:ab2}
If $b_{n+1}b_{n+2}\cdots \in {]\ba,\bb[}$ for some $n \ge 1$, then $\Omega_{\ba,\bb} = \Omega_{\ba,(b_1\cdots b_n)^\infty}$. 
\end{enumerate}
\end{lemma}

\begin{proof} 
Assume that $a_{n+1}a_{n+2}\cdots \in {]\ba,\bb[}$ for some $n \ge 1$.
Then $(a_1\cdots a_n)^\infty \prec \ba$, thus $\Omega_{(a_1\cdots a_n)^\infty,\bb} \subseteq \Omega_{\ba,\bb}$.
Suppose that there exists $i_1i_2\cdots \in \Omega_{\ba,\bb} \setminus \Omega_{(a_1\cdots a_n)^\infty,\bb}$.
Then $(a_1\cdots a_n)^\infty\prec i_{k+1}i_{k+2}\cdots \preceq \ba$ for some $k \ge 0$, i.e., $i_{k+1}\cdots i_{k+n} = a_1\cdots a_n$ and $(a_1\cdots a_n)^\infty\prec i_{k+n+1}i_{k+n+2}\cdots \preceq a_{n+1}a_{n+2}\cdots$. 
Now, $a_{n+1}a_{n+2}\cdots \in {]\ba,\bb[}$ and $i_1i_2\cdots \in
\Omega_{\ba,\bb}$ imply that $(a_1\cdots a_n)^\infty\prec i_{k+n+1}i_{k+n+2} \cdots \preceq \ba$, and we obtain inductively that $i_{k+1}i_{k+2}\cdots = (a_1\cdots a_n)^\infty$, contradicting that $i_{k+1}i_{k+2}\cdots \succ (a_1\cdots a_n)^\infty$.
This proves~(\ref{i:ab1}), and (\ref{i:ab2}) follows by symmetry.
\end{proof}

\begin{lemma} \label{l:charab}
Let $\ba = a_1a_1\cdots \preceq \bb = b_1b_2\cdots$.
Then
\[
\ell(\ba,\bb) \in \{\ba\} \cup \{(a_1\cdots a_n)^\infty : n \ge 1\},\quad r(\ba,\bb) \in \{\bb\} \cup \{(b_1\cdots b_n)^\infty : n \ge 1\}.
\]
\end{lemma}

\begin{proof}
If $\ell(\ba,\bb) \ne \ba$, then $\Omega_{\ba,\bb} = \Omega_{(a_1\cdots a_n)^\infty,\bb}$ for some $n \ge 1$ by Lemma~\ref{l:ab}~(\ref{i:ab1}), thus $\ell(\ba,\bb) = \ell((a_1\cdots a_n)^\infty,\bb)$ with $(a_1\cdots a_n)^\infty \prec \ba \preceq \bb$.
If $\ell((a_1\cdots a_n)^\infty,\bb) \ne (a_1\cdots a_n)^\infty$, then  $\Omega_{(a_1\cdots a_n)^\infty,\bb} = \Omega_{(a_1\cdots a_m)^\infty,\bb}$ for some $m < n$ by Lemma~\ref{l:ab}~(\ref{i:ab1}).
Repeating this argument and because $\ell(a_1^\infty,\bb) = a_1^\infty$ (when $a_1^\infty \preceq \bb$), we arrive at $\ell(\ba,\bb) \in \{\ba\} \cup \{(a_1\cdots a_n)^\infty : n \ge 1\}$.
The statement for $r(\ba,\bb)$ is symmetric.
\end{proof}

\begin{proof}[Proof of Lemma~\ref{l:abinOmega}]
Assume w.l.o.g.\ $\ba \notin \Omega_{\ba,\bb}$.
Then, by the proof of Lemma~\ref{l:charab}, $\ell(\ba,\bb)$ is determined by finitely many inequalities of the form $\ba \prec a_{n+1}a_{n+2}\cdots \prec \bb$, $(a_1\cdots a_n)^\infty \prec a_{m+1}\cdots a_n(a_1\cdots a_n)^\infty \prec \bb$, etc. 
Since these inequalities are strict, they are also satisfied for all $(\ba',\bb')$ in a neighbourhood of $(\ba,\bb)$, i.e., when $a'_1\cdots a'_k = a_1\cdots a_k$ and $b'_1\cdots b'_k = b_1\cdots b_k$ for sufficiently large~$k$, which implies that $\ell(\ba',\bb') \preceq \ell(\ba,\bb)$ and thus $\Omega_{\ba',\bb'} = \Omega_{\ell(\ba,\bb),\bb'}$ for these $(\ba',\bb')$. 

If $\bb \notin \Omega_{\ba,\bb}$, then we have symmetrically $r(\ba',\bb') \succeq r(\ba,\bb)$ for all $(\ba',\bb')$ in a neighbourhood of $(\ba,\bb)$.
Since $\ell(\ba,\bb), r(\ba,\bb) \in \Omega_{\ba',\bb'}$ for all $\ba' \succeq \ell(\ba,\bb)$, $\bb' \preceq r(\ba,\bb)$, with $\ell(\ba,\bb) \prec \ba$, $r(\ba,\bb) \succ \bb$, we obtain that $\Omega_{\ba',\bb'} = \Omega_{\ba,\bb}$ for all $(\ba',\bb')$ in a neighbourhood of $(\ba,\bb)$, thus $d_{q_0,q_1}$ is constant in this neighbourhood, hence continuous at $(\ba,\bb)$. 

If $\bb \in \Omega_{\ba,\bb}$, then $d_{q_0,q_1}$ is continuous at $(\ell(\ba,\bb),\bb)$ by the assumption of the lemma, thus the continuity of the projection $(\ba',\bb') \mapsto (\ell(\ba,\bb),\bb')$ and the equality $d_{q_0,q_1}(\ba',\bb') = d_{q_0,q_1}(\ell(\ba,\bb),\bb')$ in a neighbourhood of $(\ba,\bb)$ give that $d_{q_0,q_1}$ is continuous at $(\ba,\bb)$. 
\end{proof}

\begin{lemma} \label{l:phi}
Let $\varphi$ be a substitution such that $\varphi(0) \in 0\{0,1\}^*$ and $\varphi(1) \in 1\{0,1\}^*$, $\ba \in 0\{0,1\}^\infty$, $\bb \in 1\{0,1\}^\infty$.
Then
\[
0 \le d_{q_0,q_1}(\varphi(\ba),\varphi(\bb)) - d_{q_0,q_1}(\varphi(0^\infty),\varphi(1^\infty)) \le \dim_H \pi_{q_0,q_1}\big(\varphi(\Omega_{\ba,\bb})\big).
\]
In particular, $d_{q_0,q_1}(\varphi(\ba),\varphi(\bb)) = d_{q_0,q_1}(\varphi(0^\infty),\varphi(1^\infty))$ when $h(\Omega_{\ba,\bb}) = 0$.
\end{lemma}

\begin{proof}
If a sequence in $\Omega_{\varphi(01^\infty),\varphi(10^\infty)}$ starts with $\varphi(10^k)0$ or $\varphi(01^k)1$ for some $k \ge 0$, then it starts with $\varphi(10^k0)$ and $\varphi(01^k1)$ respectively. 
Recursively, this implies that the sequence is a concatenation of $\varphi(0)$ and~$\varphi(1)$, i.e.,
\[
[\varphi(0)1]_{\varphi(01^\infty),\varphi(10^\infty)} \cup [\varphi(1)0]_{\varphi(01^\infty),\varphi(10^\infty)} \subset \varphi(\{0,1\}^\infty).
\]
If $i_1i_2\cdots \in \Omega_{\varphi(01^\infty),\varphi(10^\infty)} \setminus \Omega_{\varphi(0^\infty),\varphi(1^\infty)}$, then
\[
i_ni_{n+1}\cdots \in (\varphi(0^\infty),\varphi(01^\infty)] \cup [\varphi(10^\infty),\varphi(1^\infty)) \quad \mbox{for some}\ n \ge 1
\]
and thus $i_m\cdots i_{m+|\varphi(0)|} = \varphi(0)1$ or $i_m\cdots i_{m+|\varphi(1)|} = \varphi(1)0$ for some $m \ge 1$.
This implies that $i_1i_2\cdots$ ends with a sequence in $\varphi(\{0,1\}^\infty)$.
Since $\varphi(\{0,1\}^\infty) \cap \Omega_{\varphi(\ba),\varphi(\bb)} = \varphi(\Omega_{\ba,\bb})$, we obtain that
\[
\Omega_{\varphi(\ba),\varphi(\bb)} \setminus \Omega_{\varphi(0^\infty),\varphi(1^\infty)} \subset \bigcup_{w\in\{0,1\}^*} w\varphi(\Omega_{\ba,\bb}),
\]
hence
\[
\begin{aligned}
0 & \le d_{q_0,q_1}(\varphi(\ba),\varphi(\bb)) - d_{q_0,q_1}(\varphi(0^\infty), \varphi(1^\infty)) \\
& \le \sup_{w\in\{0,1\}^*} \dim_H \pi_{q_0,q_1}\big(w\varphi(\Omega_{\ba,\bb})\big) = \dim_H \pi_{q_0,q_1}\big(\varphi(\Omega_{\ba,\bb})\big).
\end{aligned}
\]

It remains to show that $h(\Omega_{\ba,\bb}) = 0$ implies $\dim_H \pi_{q_0,q_1}\big(\varphi(\Omega_{\ba,\bb})\big) = 0$.
By Proposition~\ref{p:K}, $h(\Omega_{\ba,\bb}) = 0$ implies $\dim_H \pi_{q'_0,q'_1}(\Omega_{\ba,\bb}) = 0$ for all $q'_0, q'_1 > 1$; this also holds for $(\ba,\bb) \not\in W$ because $\Omega_{\ba,\bb} = \Omega_{\ell(\ba,\bb),r(\ba,\bb)}$ and $(\ell(\ba,\bb),r(\ba,\bb)) \in W$.
Setting $q'_0 = q_{u_1} \cdots q_{u_m}$, $d_0 = \sum_{k=1}^m u_k
\prod_{j=k+1}^m q_{u_j}$ if $\varphi(0) = u_1\cdots u_m$, and $q'_1 = q_{v_1} \cdots q_{v_n}$, $d_1 = \sum_{k=1}^n v_k \prod_{j=k+1}^n q_{u_j}$ if $\varphi(1) = v_1\cdots v_n$, we have
\[
\pi_{q_0,q_1}(\varphi(i_1i_2\cdots)) = \sum_{k=1}^\infty \frac{d_{i_k}}{q'_{i_1}\cdots q'_{i_k}} = \frac{d_0}{q'_0{-}1} + \bigg(d_1-d_0\frac{q'_1{-}1}{q'_0{-}1}\bigg) \pi_{q'_0,q'_1}(i_1i_2\cdots),
\]
where the second equality is given by \cite[Lemma~3.1]{KomSteZou2022}, thus $\dim_H \pi_{q_0,q_1}\big(\varphi(\Omega_{\ba,\bb})\big) = \dim_H \pi_{q'_0,q'_1}(\Omega_{\ba,\bb})$.
\end{proof}

The following lemma shows that $d_{q_0,q_1}$ is continuous at $(\ba,\bb)$ given by certain sequences of substitions. 

\begin{lemma} \label{l:phik}
Let $\ba \in 0\{0,1\}^\infty$, $\bb \in 1\{0,1\}^\infty$. 
If there exists a sequence of substitutions $(\varphi_k)_{k\ge1}$ such that the length of  $\varphi_k(01)$ is unbounded and
\begin{equation*} \label{e:phii}
\varphi_k(0^\infty) \prec \ba \prec \varphi_k(01^\infty),\ \varphi_k(10^\infty) \prec \bb \prec \varphi_k(1^\infty) \quad \mbox{for all}\ k \ge 0,
\end{equation*}
then $d_{q_0,q_1}$ is continuous at $(\ba,\bb)$.
\end{lemma}

\begin{proof}
By Lemma~\ref{l:phi}, we have 
\[
|d_{q_0,q_1}(\ba',\bb') - d_{q_0,q_1}(\ba,\bb)| \le s_k := \dim_H \pi_{q_0,q_1}(\varphi_k(\{0,1\}^\infty))
\]
for all $\varphi_k(0^\infty) \prec \ba' \prec \varphi_k(01^\infty)$, $\varphi_k(10^\infty) \prec \bb' \prec \varphi_k(1^\infty)$, i.e., for $(\ba',\bb')$ in a neighbourhood of $(\ba,\bb)$. 
Since $s_k$ is given by
\[
\frac{1}{(q_{u_1}\cdots q_{u_m})^{s_k}} + \frac{1}{(q_{v_1}\cdots q_{v_n})^{s_k}} = 1
\] 
where $\varphi_k(0) = u_1\cdots u_m$, $\varphi_k(1) = v_1\cdots v_n$, the condition that the length of  $\varphi_k(01)$ is unbounded implies that $s_k$ is arbritarily close to~$0$, which proves the lemma.
\end{proof}

\begin{lemma} \label{l:sclose}
Let $q_0, q_1 > 1$.
The restriction of the map $d_{q_0,q_1}$ to~$W$ is continuous.
\end{lemma}

\begin{proof}
Let $(\ba,\bb) \in W$.
Write 
\[
\tilde{K}_{\ba, \bb}(t) := \big(\pi_{q_0^t,q_1^t}(\bb) - \pi_{q_0^t,q_1^t}(\ba)\big)\,  (q_1^t{-}1)
\]
and recall that $\tilde{K}_{\ba, \bb}(t) = 1/Q_{\ba,\bb}(q_0^{-t},q_1^{-t})$ for $t > s$, where $s > 0$ is the largest root of $\tilde{K}_{\ba, \bb}(t) = 0$ (when such a root exists).
Therefore, $\tilde{K}_{\ba, \bb}(t)$ is strictly monotonically increasing for $t \ge s$. 
Moreover, the map $(\ba,\bb) \mapsto \tilde{K}_{\ba,\bb}(t)$ is continuous for all $t > 0$.
For each $\varepsilon > 0$, there is a neighbourhood~$V_\varepsilon$ of $(\ba,\bb)$ such that $\tilde{K}_{\ba',\bb'}(t) > 0$ for all $(\ba',\bb') \in V_\varepsilon$ and all $t \ge s{+}\varepsilon$.

Since $\tilde{K}_{\ba,\bb}$ is analytic, we have $\tilde{K}_{\ba,\bb}(s{-}\varepsilon) \ne 0$ for arbitrarily small $\varepsilon > 0$.
If $\tilde{K}_{\ba,\bb}(s{-}\varepsilon) < 0$, then we can assume that $V_\varepsilon$ is sufficiently small such that $|\tilde{K}_{\ba, \bb}(s{-}\varepsilon) - \tilde{K}_{\ba', \bb'}(s{-}\varepsilon)| < |\tilde{K}_{\ba, \bb}(s{-}\varepsilon)|$ for all $(\ba',\bb') \in V_\varepsilon$.
This implies that
\[
\tilde{K}_{\ba', \bb'}(s{-}\varepsilon) < 0 < \tilde{K}_{\ba', \bb'}(s{+}\varepsilon).
\]
Hence, by the intermediate value theorem, $\tilde{K}_{\ba', \bb'}(t)$ has a root $s' \in (s{-}\varepsilon, s{+}\varepsilon)$, which is the largest root of $\tilde{K}_{\ba', \bb'}(t)$.
If $\tilde{K}_{\ba,\bb}(s{-}\varepsilon) > 0$, then we can assume that $V_\varepsilon$ is sufficiently small such that
\[
|\tilde{K}_{\ba, \bb}(s{-}\varepsilon) - \tilde{K}_{\ba', \bb'}(s{-}\varepsilon)| < \frac{|\tilde{K}_{\ba, \bb}(s{-}\varepsilon)|}{2},\ |\tilde{K}_{\ba, \bb}(s) - \tilde{K}_{\ba',\bb'}(s)| < \frac{|\tilde{K}_{\ba, \bb}(s{-}\varepsilon)|}{2},
\]
for all $(\ba',\bb') \in V_\varepsilon$. 
Since $\tilde{K}_{\ba, \bb}(s) = 0$, we obtain that \[\tilde{K}_{\ba', \bb'}(s{-}\varepsilon) > \frac{\tilde{K}_{\ba, \bb}(s{-}\varepsilon)}{2} > \tilde{K}_{\ba', \bb'}(s).
\]
If $(\ba',\bb') \in W$, then $\tilde{K}_{\ba',\bb'}(t)$ is strictly increasing above the largest root~$s'$ of $\tilde{K}_{\ba', \bb'}(t) = 0$, thus $s' > s{-}\varepsilon$, and $s' < s{+}\varepsilon$ by the definition of~$V_\varepsilon$.

We have shown that, for arbitrarily small $\varepsilon > 0$, there is a neighbourhood~$V_\varepsilon$ of $(\ba,\bb)$ such that the largest root of $\tilde{K}_{\ba',\bb'}(t)$ is $\varepsilon$-close to the largest root of $\tilde{K}_{\ba,\bb}(t)$ for all $(\ba',\bb') \in V_\varepsilon \cap W$, which proves the lemma by Theorem~\ref{t:dimH}.
\end{proof}

Note that Lemma~\ref{l:sclose} only states that $d_{q_0,q_1}(\ba,\bb)$ is close to $d_{q_0,q_1}(\ba',\bb')$ when $(\ba,\bb)$ and $(\ba',\bb')$ are in~$W$. 
In view of Lemma~\ref{l:abinOmega}, it remains to consider $d_{q_0,q_1}$ at arbitrary points close to $(\ba,\bb) \in W$ for the proof of Theorem~\ref{t:t2a}.

\begin{lemma} \label{l:abcont}
Let $q_0,q_1 > 1$, $(\ba,\bb) \in W$. 
If there exists $(\ba',\bb') \in W$ arbitrarily close to $(\ba,\bb)$ with $\ba' \prec \ba$ and $\bb' \succ \bb$, then $d_{q_0,q_1}$ is continuous at $(\ba,\bb)$. 
\end{lemma}

\begin{proof}
For $(\ba',\bb') \in W$, we have $\ba' \preceq \ell(\ba'',\bb'') \preceq \ba''$ and $\bb'' \preceq r(\ba'',\bb'') \preceq \bb'$ for all $\ba'' \succeq \ba'$, $\bb'' \preceq \bb'$.
Since $\ba'$ can be chosen arbitrarily close to the left of~$\ba$ and $\bb'$ arbitrarily close to the right of~$\bb$, this implies that $\ell$ and $r$ are continuous maps at~$(\ba,\bb)$, hence $d_{q_0,q_1}$ is continuous at $(\ba,\bb)$ by Lemma~\ref{l:sclose}.
\end{proof}

\begin{proposition} \label{p:nontogether}
Let $(\ba,\bb) = (a_1a_2\cdots, b_1b_2\cdots) \in W$, with $a_{m+1} a_{m+2} \cdots \ne \bb$ and $b_{m+1} b_{m+2}\cdots \ne \ba$ for all $m \ge 1$.
Then $d_{q_0,q_1}$ is continuous at $(\ba,\bb)$. 
\end{proposition}

\begin{proof}
If $a_2 = 0$, then $d_{q_0,q_1}$ is constantly 0 around $(\ba,\bb)$ because $\Omega_{\ba',\bb'} = \Omega_{0^\infty,\bb'}$ for all $\ba' \in 00\{0,1\}^\infty$, $\bb' \in \{0,1\}^\infty$.
The case $b_2 = 1$ is symmetric.
Assume in the following that $a_2 = 1$ and $b_2 = 0$. 
Define recursively the sequence $(n_k)_{k\ge0}$ by $n_0 = 1$ and $n_k > n_{k-1}$ minimal such that $a_{n_k+1} > b_{n_k-n_{k-1}+1}$, i.e., 
\[
\ba = 0\, b_1 \cdots b_{n_1-n_0}\, b_1 \cdots b_{n_2-n_1} \cdots.
\]
For $k \ge 1$, define the substitution~$\varphi_k$ by
\[
\varphi_k(0) = a_1 \cdots a_{m_k}, \quad \varphi_k(1) = a_{m_k+1} \cdots a_{n_k},
\]
with $1 \le m_k < n_k$ satisfying
\[
(a_{m_k+1} \cdots a_{n_k} a_1\cdots a_{m_k})^\infty = \min \{(a_{i+1} \cdots a_{n_k} a_1\cdots a_i)^\infty \,:\, 1 \le i < n_k,\, a_{i+1} = 1\}.
\]
For $1 \le i < n_k$, we have thus $(a_{i+1} \cdots a_{n_k} a_1\cdots a_i)^\infty \succeq \varphi_k(10)^\infty$ when $a_{i+1} = 1$, and $(a_{i+1} \cdots a_{n_k} a_1\cdots a_i)^\infty \preceq \varphi_k(01)^\infty$ when $a_{i+1} = 0$ because $a_{i+1} \cdots a_{n_k} a_1 \prec a_{i+1} \cdots a_{n_k+1} \preceq a_1 \cdots a_{n_k-i+1}$.
This implies that
\[
(\varphi_k(01)^\infty, \varphi_k(10)^\infty) \in W.
\]
Moreover, we have $m_k = n_j$ for some $0 \le j < k$.
Indeed, for $n_j < i < n_{j+1}$, $0 \le j < k$, with $a_{i+1} = 1$, we have
\begin{equation} \label{e:inj}
(a_{i+1} \cdots a_{n_k} a_1\cdots a_i)^\infty \succ (a_{n_j+1} \cdots a_{n_k} a_1\cdots a_{n_j})^\infty
\end{equation}
because
\begin{equation} \label{e:anj}
\begin{aligned}
a_{i+1} \cdots a_{n_{j+1}+1} & = b_{i-n_j+1} \cdots b_{n_{j+1}-n_j} 1 \succ b_{i-n_j+1} \cdots b_{n_{j+1}-n_j+1} \\
& \succeq b_1 \cdots b_{n_{j+1}-i+1} = a_{n_j+1} \cdots a_{n_j+n_{j+1}-i+1},
\end{aligned}
\end{equation}
which immediately gives \eqref{e:inj} when $j \le k{-}2$ and which, for $j = k{-}1$, implies that $a_{i+1} \cdots a_{n_k} a_1 \succeq a_{n_{k-1}+1} \cdots a_{n_{k-1}+n_k-i+1}$, thus
\[
a_{i+1} \cdots a_{n_k} a_1 \cdots a_{i-n_{k-1}+1} \succeq a_{n_{k-1}+1} \cdots a_{n_k+1} \succ a_{n_{k-1}+1} \cdots a_{n_k} a_1.
\]

Since $a_1 < a_{n_k+1}$ implies that $\varphi_k(01)^\infty \prec \ba$, and $\lim_{k\to\infty} \varphi_k(01)^\infty = \ba$, pairs $(\varphi_k(01)^\infty, \varphi_k(10)^\infty)$ satisfy the assumptions of Lemma~\ref{l:abcont} when $\varphi_k(10)^\infty$ is arbitrarily close to the right of~$\bb$.

Symmetrically, define the substitution~$\varphi'_k$, $k \ge 1$, by
\[
\varphi'_k(0) = b_{m'_k+1} \cdots b_{n'_k}, \quad \varphi'_k(1) = b_1 \cdots b_{m'_k},
\]
with the sequence $(n'_k)_{k\ge0}$ satisfying $n'_0 = 1$, $n'_k > n'_{k-1}$ minimal such that $b_{n'_k+1} < a_{n'_k-n'_{k-1}+1}$, and $1 \le m'_k < n'_k$ such that $\varphi'_k(01)^\infty \in \Omega_{\varphi'_k(01)^\infty,\varphi'_k(10)^\infty}$.
Then pairs $(\varphi'_k(01)^\infty, \varphi'_k(10)^\infty)$ satisfy the assumptions of Lemma~\ref{l:abcont} when $\varphi'_k(01)^\infty$ is arbitrarily close to the left of~$\ba$.

If $\varphi_k(10)^\infty$ is bounded away from the right of~$\bb$ for infinitely many~$k$, and $\varphi'_j(01)^\infty$ is bounded away from the left of~$\ba$ for infinitely many~$j$, then we obtain
\[
(\varphi_k(01)^\infty, \varphi'_j(10)^\infty) \in W
\]
for arbitarily large $n_k,n'_j$, thus we can apply again Lemma~\ref{l:abcont}.

It remains to consider the case that $\varphi_k(10)^\infty \preceq \bb$ for almost all~$k$, the case $\varphi'_k(01)^\infty \succeq \ba$ for almost all~$k$ being symmetric.
Since $m_k = n_j$ for some $j < k$, we have $\varphi_k(1^\infty) \succ \bb$. 
Moreover, $\varphi_{k+1}(10)^\infty \preceq \bb$ implies that $m_{k+1} = n_k$, thus $\ba$ starts with $\varphi_k(01) \varphi_{k+1}(1)$, and we have $\ba \prec \varphi_k(01^\infty)$ when $|\varphi_k(1)| < |\varphi_{k+1}(1)|$.
When $\ba \ne \varphi_k(01^\infty)$ for all $k \ge 1$, there are infinitely many~$k$ such that $|\varphi_k(1)| < |\varphi_{k+1}(1)|$, thus $\varphi_k(0^\infty) \prec \ba \prec \varphi_k(01^\infty)$ and $\varphi_k(10^\infty) \prec \bb \prec \varphi_k(1^\infty)$ for infinitely many~$k$, and $d_{q_0,q_1}$ is continuous at $(\ba,\bb)$ by Lemma~\ref{l:phik}. 

Finally, suppose that $\ba = \varphi_k(01^\infty)$ for some $k \ge 1$.
Then $\varphi_{k+1}(0) = \varphi_k(01)$, $\varphi_{k+1}(1) = \varphi_k(1)$, and, inductively, $\varphi_{k+j}(0) = \varphi_k(01^j)$, $\varphi_{k+j}(1) = \varphi_k(1)$, for all $j \ge 0$.
Since $\varphi_{k+j}(10)^\infty \preceq \bb$ for almost all~$j$, this would imply that $\varphi_k(101^\infty) \preceq \bb$, thus $b_{|\varphi_k(1)|+1} b_{|\varphi_k(1)|+2} \cdots \succeq \ba$, which contradicts the assumptions $\bb \in \Omega_{\ba,\bb}$ and $b_{m+1} b_{m+2} \cdots \ne \ba$ for all $m \ge 1$. 
This proves the proposition. 
\end{proof}

For the continuity of $d_{q_0,q_1}$, it only remains to prove the following proposition.

\begin{proposition} \label{p:together}
Let $(\ba,\bb) = (a_1a_2\cdots, b_1b_2\cdots) \in W$ with $a_{m+1} a_{m+2} \cdots = \bb$ or $b_{m+1} b_{m+2} \cdots = \ba$ for some $m \ge 1$.
Then $d_{q_0,q_1}$ is continuous at $(\ba,\bb)$. 
\end{proposition}

\begin{proof}
Assume that $b_{m+1} b_{m+2} \cdots = \ba$, the other case being symmetric.
Moreover, we assume w.l.o.g.\ that $m$ is minimal, i.e., $b_{i+1} b_{i+2} \cdots \ne \ba$ for all $i < m$.

If $a_{i+1}a_{i+2}\cdots = \bb$ for some $i \ge 1$, then $\ba = \varphi(01)^\infty$, $\bb = \varphi(10)^\infty$ with $\varphi(0) = a_1\cdots a_i$, $\varphi(1) = b_1\cdots b_m$.
Since $h(\Omega_{0110^\infty,1001^\infty}) = h(\Omega_{(01)^\infty,(10)^\infty}) = 0$ and $\ba \prec \varphi(0110^\infty)$, $\bb \succ \varphi(1001^\infty)$, $d_{q_0,q_1}$ is constant in a neighbourhood of $(\ba,\bb)$ by Lemma~\ref{l:phi}, thus continuous at $(\ba,\bb)$. 

Assume now that $a_{i+1}a_{i+2}\cdots \ne \bb$ for all $i \ge 1$.
Then we can define $n_k$ as in the proof of Proposition~\ref{p:nontogether}.
For $k \ge 0$, define the substitution~$\varphi'_k$ by
\[
\varphi'_k(0) = a_1 \cdots a_{n_k}, \quad \varphi'_k(1) = b_1 \cdots b_m.
\]
We claim that, for all $k \ge 1$,
\begin{equation} \label{e:phi'}
\big(\varphi'_k(01^\infty), \varphi'_k(1^\infty)\big) \in W.
\end{equation}

To prove the claim, note first that the only sequence of $\Omega_{\ba,\bb}$ starting with $\varphi'_k(1)0$ is~$\bb$.
Since $a_{n+1}a_{n+2}\cdots \ne \bb$ for all $n \ge 1$ and $\ba \in \Omega_{\ba,\bb}$, this implies that $\ba \in \Omega_{\ba,\varphi'_k(1^\infty)}$, thus 
\begin{equation} \label{e:ai}
a_1 \cdots a_i \varphi'_k(1^\infty) \preceq \ba \quad \mbox{for all $i \ge 1$ such that $a_{i+1} = 1$.}
\end{equation}
In particular, we have
\[
\varphi'_k(01^\infty) \preceq \ba
\]
and thus $n_{k+1}{-}n_k \le m$ for all $k \ge 0$.

We strengthen now \eqref{e:ai} to 
\begin{equation} \label{e:ai2}
a_1 \cdots a_i \varphi'_k(1^\infty) \preceq \varphi'_k(01^\infty) \quad \mbox{for all $1 \le i \le n_k$ such that $a_{i+1} = 1$,}
\end{equation}
and, equivalently, 
\begin{equation} \label{e:ai3}
a_{i+1} \cdots a_{n_k} \varphi'_k(1^\infty) \succeq \varphi'_k(1^\infty) \quad \mbox{for all $1 \le i < n_k$ such that $a_{i+1} = 1$.}
\end{equation}
Let $n_j \le i < n_{j+1}$, $0 \le j < k$, $a_{i+1} = 1$.
Then $a_{i+1} \cdots a_{n_{j+1}+1} \succ b_1 \cdots b_{n_{j+1}-i+1}$ by~\eqref{e:anj}, and $n_{j+1}{-}i \le n_{j+1}{-}n_j \le m$ implies that $a_{i+1} \cdots a_{n_{j+1}} \varphi'_k(1^\infty) \succeq \varphi'_k(1^\infty)$.
Recursively from $j = k{-}1$ to $j = 0$, we obtain that
\[
\begin{aligned}
a_{i+1} \cdots a_{n_k} \varphi'_k(1^\infty) & = a_{i+1} \cdots a_{n_{j+1}} a_{n_{j+1}+1} \cdots a_{n_k} \varphi'_k(1^\infty) \\
& \succeq a_{i+1} \cdots a_{n_{j+1}} \varphi'_k(1^\infty) \succeq \varphi'_k(1^\infty).
\end{aligned}
\]
This proves \eqref{e:ai3} and thus \eqref{e:ai2}.

To show that $\varphi'_k(1^\infty) \in \Omega_{\varphi'_k(01^\infty),\varphi'_k(1^\infty)}$, let $1 \le i < m$.
If $b_{i+1} = 1$, then $b_{i+1} \cdots b_m b_1 \succ b_{i+1} \cdots b_{m+1} \succeq b_1 \cdots b_{m-i+1}$, thus $b_{i+1} \cdots b_m \varphi'_k(1^\infty) \succ \varphi'_k(1^\infty)$.
If $b_{i+1} = 0$, then $b_{i+1} \cdots b_{m+1} \preceq a_1 \cdots a_{m-i+1}$, with $a_{m-i+1} = 1$ in case $b_{i+1} \cdots b_m = a_1 \cdots a_{m-i}$ because $a_{m-i+1} = 0$ would imply $\ba \preceq a_1 \cdots a_{m-i} \ba = b_{i+1} b_{i+2} \cdots \preceq \ba$, thus $b_{i+1} b_{i+2} \cdots = \ba$, contradicting the minimality of~$m$.
By \eqref{e:ai2}, which we can use because $m < m{+}n_{k-1} \le n_k$, we obtain that $b_{i+1} \cdots b_m \varphi'_k(1^\infty) \preceq \varphi'_k(01^\infty)$, thus $\varphi'_k(1^\infty) \in \Omega_{\varphi'_k(01^\infty),\varphi'_k(1^\infty)}$.
To prove \eqref{e:phi'}, it remains to consider $a_{i+1} \cdots a_{n_k} \varphi'_k(1^\infty)$ for $1 \le i < n_k$.
If $a_{i+1} = 0$, then $a_{i+1} \cdots a_{n_k+1} \preceq a_1 \cdots a_{n_k-i+1}$,  
$a_{n_k+1} = 1$, and \eqref{e:ai2} give that $a_{i+1} \cdots a_{n_k} \varphi'_k(1^\infty) \preceq \varphi'_k(01^\infty)$.
By \eqref{e:ai3}, we obtain $\varphi'_k(01^\infty) \in \Omega_{\varphi'_k(01^\infty),\varphi'_k(1^\infty)}$, i.e., \eqref{e:phi'} holds.

Consider now $\ba' \succeq \varphi'_k(01^\infty)$, $\bb' \preceq \varphi'_k(1^\infty)$.
Then $\ell(\ba',\bb') \succeq \varphi'_k(01^\infty)$ and $r(\ba',\bb') \preceq \varphi'_k(1^\infty)$.
In particular, $\ell(\ba',\bb')$ is close to~$\ba'$ for $\ba'$ close to the right of $\varphi'_k(01^\infty)$.
Since $b_{i+1} \cdots b_m \ba \succ \bb$ for all $1 \le i < m$ such that $b_{i+1} = 1$, we have $b_{i+1} \cdots b_m \ell(\ba',\bb') \succeq \bb'$ for all such~$i$ when $\ell(\ba',\bb')$ is close to~$\ba$ and $\bb'$ is close to~$\bb$.
Since $b_{i+1} \cdots b_m \ell(\ba',\bb') \preceq b_{i+1} \cdots b_m \varphi'_k(1^\infty) \preceq \varphi'_k(01^\infty)$ for $b_{i+1} = 0$, this implies that $b_{i+1} \cdots b_m \ell(\ba',\bb') \notin {]\ba',\bb'[}$ for all $1 \le i < m$.
If $\bb' \preceq \varphi'_k(1) \ell(\ba',\bb')$, then we obtain that $\varphi'_k(1) \ell(\ba',\bb') \in \Omega_{\ba',\bb'}$ and thus $r(\ba',\bb') \preceq \varphi'_k(1) \ell(\ba',\bb')$, hence $r(\ba',\bb')$ is close to~$\bb'$, and $d_{q_0,q_1}(\ba',\bb')$ is close to $d_{q_0,q_1}(\ba,\bb)$ by Lemma~\ref{l:sclose}.
If $\bb' \succ \varphi'_k(1) \ell(\ba',\bb')$, then Lemma~\ref{l:ab} gives that $r(\ba',\bb') = \varphi'_k(1^\infty)$.
Since $\Omega_{\ba,\bb} \setminus \Omega_{\ba,\varphi'_k(1^\infty)}$ contains only sequences ending with~$\bb$, we have $d_{q_0,q_1}(\ba,\bb) = d_{q_0,q_1}(\ba,\varphi'_k(1^\infty))$, and again $d_{q_0,q_1}(\ba',\bb')$ is close to $d_{q_0,q_1}(\ba,\bb)$ by Lemma~\ref{l:sclose}.
This proves that, for large~$k$, $d_{q_0,q_1}(\ba',\bb')$ is close to $d_{q_0,q_1}(\ba,\bb)$ for all $\ba'$ close to the right of $\varphi'_k(01^\infty)$ and close to~$\bb$.
If $\varphi'_k(01^\infty) \ne \ba$ for all~$k$, i.e., $\varphi'_k(01^\infty) \nearrow \ba$, then this proves the continuity of $d_{q_0,q_1}$ at $(\ba,\bb)$.
If $\varphi'_k(01^\infty) = \ba$ for some $k \ge 1$, then this holds for all sufficiently large~$k$, and the proof of Lemma~\ref{l:phik} with the substitutions~$\varphi'_k$ shows that $d_{q_0,q_1}(\ba',\bb')$ is close to $d_{q_0,q_1}(\ba,\bb)$ for all $\ba'$ close to the left of~$\ba$ and close to~$\bb$.
Therefore, $d_{q_0,q_1}$ is continuous at $(\ba,\bb)$ also in this case.
\end{proof}

\begin{proof}[Proof of Theorem~\ref{t:t2a}]
This follows from Lemma~\ref{l:abinOmega}, Propositions~\ref{p:nontogether} and~\ref{p:together}.
\end{proof}

The following examples illustrate the different cases of Propositions~\ref{p:nontogether} and~\ref{p:together}.

\begin{example}\label{ex1}
Let $\ba = 01101(10)^21(10)^3\cdots$ and $\bb = (10)^\infty$.
Then $(\ba, \bb) \in W$, and the substitution~$\varphi_k$ of Proposition~\ref{p:nontogether} is given by $\varphi_k(0) = 01101\cdots(10)^{k-2}1$ and $\varphi_k(1) = (10)^{k-1}1$ for $k \ge 1$, i.e., $n_k = k^2 {+} 1$ and $m_k = (k{-}1)^2 {+} 1$.
We have $\varphi_k((01)^\infty) \nearrow \ba$, $\varphi_k((10)^\infty) \searrow \bb$, and $(\varphi_k((01)^\infty), \varphi_k((10)^\infty) \in W$.
By Lemma~\ref{l:abcont}, $d_{q_0,q_1}$ is continuous at $(\ba,\bb)$. 
\end{example}

\begin{example}\label{ex2}
Let $\ba = (011)^\infty$ and $\bb = (10)^\infty$.
Then $(\ba, \bb) \in W$, and the substitution~$\varphi_k$ of Proposition~\ref{p:nontogether} is given by $\varphi_k(0) = 01(101)^{k-2}$ and $\varphi_k(1) = 101$ for $k \ge 2$, i.e., $n_k = 3k{-}1$ and $m_k = 3k{-}4$, hence $\varphi_k((10)^\infty) = (10(101)^{k-1})^\infty$ is bounded away from the right of~$\bb$ for $k \ge 2$.
The substitution~$\varphi'_k$ of Proposition~\ref{p:nontogether} is given by $\varphi'_k(0) = (01)^k$ and $\varphi'_k(1) = 1$ for $k \ge 1$, i.e., $n'_k = 2k{+}1$ and $m'_k = 1$, and $\varphi'_k((01)^\infty) = ((01)^k1)^\infty$ is bounded away from the left of~$\ba$ for $k \ge 2$.
We have thus $\varphi_k((01)^\infty) \nearrow \ba$, $\varphi'_k((10)^\infty) \searrow \bb$, and
\[
\big(\varphi_k((01)^\infty), \varphi'_k((10)^\infty)\big) = \big(((011)^{k-1}01)^\infty, ((10)^k1\big)^\infty) \in W \quad \mbox{for all}\ k \ge 2. 
\]
By Lemma~\ref{l:abcont}, $d_{q_0,q_1}$ is continuous at $(\ba,\bb)$. 
\end{example}

\begin{example} \label{ex3}
Let $\ba = 01010010\cdots$ and $\bb = 1001010 \cdots$ be the fixed points of the substitution~$\psi$ defined by $\psi(0) = 010$, $\psi(1) = 10$.
Then $\ba$ is the largest element of the Fibonacci shift starting with~$0$ and $\bb$ is the smallest element of the Fibonacci shift starting with~$1$; see e.g.~\cite{KomSteZou2022}.
Therefore, we have $(\ba,\bb) \in W$.
The substitution~$\varphi_k$ of Proposition~\ref{p:nontogether} is given by $\varphi_k(0) = \psi^{k-1}(0)$ and $\varphi_k(1) = \psi^k(1) = \psi^{k-1}(10)$, hence $\varphi_k((10)^\infty) = \psi^{k-1}((100)^\infty) \prec \psi^{k-1}(\bb) = \bb$ for all $k \ge 1$.
Symmetrically, we obtain that $\varphi'_k((01)^\infty) \succ \ba$ for all $k \ge 1$ for the substitution~$\varphi'_k$ of Proposition~\ref{p:nontogether}.
Since $\varphi_k(0^\infty) = \psi^{k-1}(0^\infty) \prec \ba \prec \psi^{k-1}((01)^\infty) = \varphi_k(01^\infty)$ and $\varphi_k(10^\infty) = \psi^{k-1}(10^\infty) \prec \bb \prec \psi^{k-1}((10)^\infty) = \varphi_k(1^\infty)$, we can apply Lemma~\ref{l:phik} to see that $d_{q_0,q_1}$ is continuous at $(\ba,\bb)$.
\end{example}

\begin{example}\label{ex4}
Let $\ba=(01110)^\infty$ and $\bb=(10011)^\infty$.
Then $(\ba,\bb) \in W$, and $\ba = \varphi((01)^\infty)$, $\bb = \varphi((10)^\infty)$ with $\varphi(0) = 011$, $\varphi(1) = 10$.
By Lemma~\ref{l:phi}, we have $d_{q_0,q_1}(\varphi(0^\infty),\varphi(1^\infty)) = d_{q_0,q_1}(\varphi(0110(01)^\infty),\varphi(1001(10)^\infty))$, thus $d_{q_0,q_1}$ is constant around $(\ba,\bb)$. 
\end{example}

\begin{example}\label{ex5}
Let $\ba=(011)^\infty$ and $\bb=10(011)^\infty$.
Then $(\ba, \bb) \in W$ and the substitution~$\varphi'_k$ of Proposition~\ref{p:together} is given by $\varphi'_k(1) = 10$, $\varphi'_{2k-1}(0) = (011)^{k-1}01$, $\varphi'_{2k}(0) = (011)^k0$, for all $k \ge 1$.
We have $((011)^k(01)^\infty,(10)^\infty) \in W$ for all $k \ge 0$.
Let $(011)^k(01)^\infty \preceq \ba' \preceq (011)^k(10)^\infty$.
For $10(011)^k(01)^\infty \preceq \bb' \preceq 10\ba'$, we have $(011)^k(01)^\infty \preceq \ell(\ba',\bb') \preceq \ba'$ and $r(\ba',\bb') = 10\ell(\ba',\bb')$, thus $(\ell(\ba',\bb'), r(\ba',\bb'))$ is close to $(\ba,\bb)$ for large~$k$, hence $d_{q_0,q_1}(\ell(\ba',\bb'), r(\ba',\bb'))$ is close to $d_{q_0,q_1}(\ba,\bb)$.
For $10\ba' \prec \bb' \preceq (10)^\infty$, we have $(011)^k(01)^\infty \preceq \ell(\ba',\bb') \preceq \ba'$ and $r(\ba',\bb') = (10)^\infty$, thus $(\ell(\ba',\bb'), r(\ba',\bb'))$ is close to $(\ba,(10)^\infty)$ for large~$k$, hence $d_{q_0,q_1}(\ba',\bb')$ is close to $d_{q_0,q_1}(\ba,(10)^\infty) = d_{q_0,q_1}(\ba,\bb)$.
\end{example}

\begin{example}\label{ex6}
Let $\ba = 011(10)^\infty$ and $\bb = 10011(10)^\infty$.
Then $(\ba, \bb) \in W$ and the substitution~$\varphi'_k$ of Proposition~\ref{p:together} is given by $\varphi'_k(0) = 011(10)^{k-2}$, $\varphi'_k(1) = 10$, for all $k \ge 2$.
We have thus $\ba = \varphi'_k(01^\infty)$ for all $k \ge 2$.
Let $(\ba',\bb')$ be close to $(\ba,\bb)$.
If $\ba' \succeq \ba$ and $\bb' \preceq 10\ba'$, then $(\ell(\ba',\bb'), r(\ba',\bb'))$ is close to $(\ba,\bb)$, hence $d_{q_0,q_1}(\ba',\bb')$ is close to $d_{q_0,q_1}(\ba,\bb)$.
If $\ba' \succeq \ba$ and $\bb' \succ 10\ba'$, then $\ell(\ba',\bb')$ is close to~$\ba$ and $r(\ba',\bb') = (10)^\infty$, hence $d_{q_0,q_1}(\ba',\bb')$ is close to $d_{q_0,q_1}(\ba,(10)^\infty) = d_{q_0,q_1}(\ba,\bb)$.
If $\varphi'_k(0^\infty) \preceq \ba' \preceq \varphi'_k(01^\infty) = \ba$ and $\varphi'_k(10^\infty) \preceq \bb' \preceq \varphi'_k(1^\infty)$ for some large~$k$, then $d_{q_0,q_1}(\ba',\bb')$ is close to $d_{q_0,q_1}(\ba,\bb)$ by the proof of Lemma~\ref{l:phik}.
\end{example}

\begin{lemma} \label{l:qclose}
Let $\ba \in 0\{0,1\}^\infty\!$, $\bb \in 0\{0,1\}^\infty\!$.
Then the map $(q_0,q_1) \mapsto d_{q_0,q_1}(\ba,\bb)$ is continuous for $q_0,q_1 > 1$.
\end{lemma}

\begin{proof}
Since $d_{q_0,q_1}(\ba,\bb) = d_{q_0,q_1}(\ell(\ba,\bb),r(\ba,\bb))$, we can assume w.l.o.g.\ $(\ba,\bb) {\,\in\,} W$.
Then the proof is very similar to that of Lemma~\ref{l:sclose}, by considering
\[
\hat{K}_{q_0, q_1}(t) := \big(\pi_{q_0^t,q_1^t}(\bb) - \pi_{q_0^t,q_1^t}(\ba)\big)\,  (q_1^t{-}1)
\]
instead of the function $\tilde{K}_{\ba,\bb}$ from the proof of Lemma~\ref{l:sclose}.
\end{proof}

\begin{proof}[Proof of Theorem~\ref{t:t2}]
Let $q_0,q_1,q'_0,q'_1>1$. 
Recall that the quasi-greedy $(q_0,q_1)$-expansion of a number $x \in [0,1/(q_1{-}1]$ is the sequence $i_1i_2\cdots$ such that $i_n=1$ if $T_{i_{n-1}} \circ \cdots \circ T_{i_1}(x) > 1/q_1$, $i_n = 0$ otherwise, with $T_i$ defined in~\eqref{e:Ti}.
Let $\ba_{q_0,q_1} = a_1a_2\cdots$. 
If $T_{a_n} \circ \cdots \circ T_{a_1}(1/q_1) \ne 1/q_1$ for all $n \ge 1$, then $\ba_{q'_0,q'_1}$ is close to $\ba_{q_0,q_1}$ when $(q'_0,q'_1)$ is close to $(q_0,q_1)$.
If $\ba_{q_0,q_1} = (a_1\cdots a_m)^\infty$, and $m$ is minimal with this property, then $\ba_{q'_0,q'_1}$ is close to $\ba_{q_0,q_1}$ or to $a_1\cdots a_m 10^\infty$. 
In the latter case, we have $\Omega_{\ba_{q'_0,q'_1},\bb_{q'_0,q'_1}} = \Omega_{\ba_{q_0,q_1},\bb_{q'_0,q'_1}}$ by Lemma~\ref{l:ab}.
Similarly, $\bb_{q'_0,q'_1}$ is close to $\bb_{q_0,q_1}$ or to $b_1\cdots b_n 01^\infty$, with $\Omega_{\ba_{q'_0,q'_1},\bb_{q'_0,q'_1}} = \Omega_{\ba_{q'_0,q'_1},(b_1\cdots b_n)^\infty}$. 
Therefore, $d_{q_0,q_1}(\ba_{q_0,q_1},\bb_{q_0,q_1})$ is close to $d_{q_0,q_1}(\ba_{q'_0,q'_1},\bb_{q'_0,q'_1})$ by Theorem~\ref{t:t2a}, thus close to $d_{q'_0,q'_1}(\ba_{q'_0,q'_1},\bb_{q'_0,q'_1})$ by Lemma~\ref{l:qclose}.
\end{proof}

It only remains to prove Theorem~\ref{t:equaldim}, using the following results.

\begin{lemma} \label{l:uvtilde}
Let $\bu \in 0\{0,1\}^*$, $\bv \in 1\{0,1\}^*$.
Then
\[
\Omega_{\tilde{\bu}^\infty,\tilde{\bv}^\infty} = \Omega_{\bu^\infty,\bv^\infty} \subseteq \Omega_{\bu\bv^\infty,\bv\bu^\infty} \subseteq \Omega_{\tilde{\bu}\tilde{\bv}^\infty,\tilde{\bv}\tilde{\bu}^\infty}.
\]
for some prefix~$\tilde{\bu}$ of~$\bu$ and some prefix~$\tilde{\bv}$ of~$\bv$ with $\tilde{\bu}^\infty, \tilde{\bv}^\infty \in \Omega_{\tilde{\bu}^\infty,\tilde{\bv}^\infty}$.
\end{lemma}

\begin{proof}
By Lemma~\ref{l:charab}, we have $\ell(\bu^\infty,\bv^\infty) = \tilde{\bu}^\infty$, $r(\bu^\infty,\bv^\infty) = \tilde{\bv}^\infty$ for some prefix~$\tilde{\bu}$ of~$\bu$ and some prefix~$\tilde{\bv}$ of~$\bv$.
Let $\tilde{\bu}, \tilde{\bv}$ be the shortest such prefixes.
Then $\tilde{\bu}^\infty, \tilde{\bv}^\infty \in \Omega_{\tilde{\bu}^\infty,\tilde{\bv}^\infty} = \Omega_{\bu^\infty,\bv^\infty} \subseteq \Omega_{\bu\bv^\infty,\bv\bu^\infty}$, and it only remains to prove that $\Omega_{\bu\bv^\infty,\bv\bu^\infty} \subseteq \Omega_{\tilde{\bu}\tilde{\bv}^\infty,\tilde{\bv}\tilde{\bu}^\infty}$.
Suppose that $\bu\bv^\infty \succ \tilde{\bu} \tilde{\bv}^\infty$.
Since $\tilde{\bv}^\infty \in \Omega_{\bu^\infty,\bv^\infty}$, we cannot have $\bu^\infty \prec \tilde{\bu} \tilde{\bv}^\infty \prec \bu\bv^\infty$, thus we must have $\tilde{\bu} \tilde{\bv}^\infty \preceq \bu^\infty$.
Since $\tilde{\bu} \tilde{\bv}^\infty \succ \tilde{\bu}^\infty = \ell(\bu^\infty,\bv^\infty)$, we have $\tilde{\bu} \tilde{\bv}^\infty \notin \Omega_{\bu^\infty,\bv^\infty}$, hence $u_{n+1} \cdots u_{|\tilde{\bu}|} \tilde{\bv}^\infty \in (\bu^\infty, \bv^\infty)$ for some $1 \le n < |\tilde{\bu}|$, where we write $\bu = u_1 \cdots u_{|\bu|}$.
Since $\tilde{\bu}^\infty \in \Omega_{\bu^\infty,\bv^\infty}$, this implies that $u_{n+1} = 0$.
Then $\tilde{\bu} \tilde{\bv}^\infty \preceq \bu^\infty$ implies that $u_{n+1} \cdots u_{|\bu|} \bu^\infty \in (\bu^\infty, \bv^\infty)$.
By the proof of Lemma~\ref{l:charab}, we obtain that the period of $\ell(\bu^\infty,\bv^\infty)$ is at most~$n$, contradicting that it is~$|\tilde{\bu}|$.
Therefore, we have $\bu\bv^\infty \preceq \tilde{\bu} \tilde{\bv}^\infty$.
Symmetrically, we obtain that $\tilde{\bv}\tilde{\bu}^\infty \preceq \bv\bu^\infty$, thus $\Omega_{\bu\bv^\infty,\bv\bu^\infty} \subseteq \Omega_{\tilde{\bu}\tilde{\bv}^\infty,\tilde{\bv}\tilde{\bu}^\infty}$.
\end{proof}

\begin{lemma} \label{l:notuvprime}
Let $\bu = \tilde{\bu}\tilde{\bv}$, $\bv = \tilde{\bv}\tilde{\bu}$ for some $\tilde{\bu} \in 0\{0,1\}^*$, $\tilde{\bv} \in 1\{0,1\}^*$.
Then
\begin{equation} \label{e:Omegauv}
\Omega_{\tilde{\bu}^\infty,\tilde{\bv}^\infty} \subseteq \Omega_{\bu^\infty,\bv^\infty} \subseteq \Omega_{\bu\bv^\infty,\bv\bu^\infty} \subseteq \Omega_{\tilde{\bu}\tilde{\bv}^\infty,\tilde{\bv}\tilde{\bu}^\infty}
\end{equation}
and 
$d_{q_0,q_1}(\bu^\infty,\bv^\infty) = d_{q_0,q_1}(\tilde{\bu}^\infty,\tilde{\bv}^\infty)$ for all $q_0,q_1 > 1$.
\end{lemma}

\begin{proof}
We obviously have $\tilde{\bu}^\infty \preceq \bu^\infty \prec \bu\bv^\infty \preceq \tilde{\bu} \tilde{\bv}^\infty$ and $\tilde{\bv}\tilde{\bu}^\infty \preceq \bv\bu^\infty \prec \bv^\infty \preceq \tilde{\bv}^\infty$, which implies \eqref{e:Omegauv}.
The equality $d_{q_0,q_1}(\bu^\infty,\bv^\infty) = d_{q_0,q_1}(\tilde{\bu}^\infty,\tilde{\bv}^\infty)$ follows from Lemma~\ref{l:phi} because $h(\Omega_{(01)^\infty,(10)^\infty}) = 0$.
\end{proof}

\begin{proposition} \label{p:uv}
Let $\bu \in 01\{0,1\}^*$, $\bv \in 10\{0,1\}^*$ be such that $\bu^\infty, \bv^\infty \in \Omega_{\bu^\infty,\bv^\infty}$ and there is no decomposition $\bu = \tilde{\bu}\tilde{\bv}$, $\bv = \tilde{\bv}\tilde{\bu}$.
Then $d_{q_0,q_1}(\bu^\infty,\bv^\infty) = d_{q_0,q_1}(\bu\bv^\infty,\bv\bu^\infty)$ for all $q_0,q_1 > 1$.
\end{proposition}

\begin{proof}
Assume first that $\bu = \tilde{\bu}^k$ for some $\tilde{\bu} \in 01\{0,1\}^*$, $k \ge 2$.
Then Lemma~\ref{l:phi} implies that $d_{q_0,q_1}(\tilde{\bu}^\infty,\bv^\infty) = d_{q_0,q_1}(\tilde{\bu}^k\bv^\infty,\bv\tilde{\bu}^\infty)$ because $h(\Omega_{0^k\ba,\bb}) = h(\Omega_{0^\infty,\bb}) = 0$ for all $\ba,\bb$, thus $d_{q_0,q_1}(\bu^\infty,\bv^\infty) = d_{q_0,q_1}(\bu\bv^\infty,\bv\bu^\infty)$.
Symmetrically, the proposition holds when $\bv = \tilde{\bv}^k$ for some $\tilde{\bv} \in 10\{0,1\}^*$, $k \ge 2$.

Assume in the following that $\bu$ and $\bv$ are not powers of smaller words, and write $\bu^\infty = u_1 u_2 \cdots$, $\bv^\infty = v_1 v_2 \cdots$.
Then the absence of a decomposition $\bu = \tilde{\bu}\tilde{\bv}$, $\bv = \tilde{\bv}\tilde{\bu}$ implies that $u_{k+1} u_{k+2} \cdots \ne \bv^\infty$ and $v_{k+1} v_{k+2} \cdots \ne \bu^\infty$ for all $k \ge 0$.
Let $m,n \ge 1$ be maximal such that 
\[
\begin{aligned}
u_{k+1} \cdots u_{k+m} & = v_1 \cdots v_m \quad \mbox{for some}\ 1 \le k < |\bu|, \\
 v_{k+1} \cdots v_{k+n} & = u_1 \cdots u_n \quad \mbox{for some}\ 1 \le k < |\bv|,
\end{aligned}.
\]
Then we have, as in the proof of \cite[Lemma~8]{BarSteVin2014},
\[
\{u_1\cdots u_n,v_1 \cdots v_m\}^\infty \subseteq \Omega_{\bu^\infty,\bv^\infty},
\]
thus 
\begin{equation} \label{e:dimuv}
d_{q_0,q_1}(\bu^\infty,\bv^\infty) \ge \dim_H \pi_{q_0,q_1}\big(\{u_1\cdots u_n,v_1\cdots v_m\}^\infty\big) = \min\{s,1\}, 
\end{equation}
with
\[
\frac{1}{q_{u_1}^s\cdots q_{u_n}^s} + \frac{1}{q_{v_1}^s\cdots q_{v_m}^s} = 1.
\]
Suppose that $n > |\bv|$.
Then $v_{k+1} \cdots v_{k+|\bv|} = u_1 \cdots u_{|\bv|}$ and $u_{|\bv|+1}  = v_{k+|\bv|+1} = v_{k+1} = u_1 = 0$. 
Then $\bv^\infty \in \Omega_{\bu^\infty,\bv^\infty}$ implies that $(u_1 \cdots u_{|\bv|})^\infty \preceq \bu^\infty$, thus
\[
(u_1 \cdots u_{|\bv|})^\infty \preceq u_{|\bv|+1} u_{|\bv|+2} \cdots \preceq \bu^\infty,
\]
hence $u_{|\bv|+1} \cdots u_{2|\bv|} = u_1 \cdots u_{|\bv|}$.
Inductively, we obtain that $(u_1 \cdots u_{|\bv|})^\infty = \bu^\infty$, i.e., $\bu = \tilde{\bu} \tilde{\bv}$, $\bv = \tilde{\bv} \tilde{\bu}$ with $\tilde{\bu} = u_1 \cdots u_{|\bv|-k}$, $\tilde{\bv} = v_1 \cdots v_k$, contradicting our assumption.
This implies that $n \le |\bv|$, hence $q_{u_1} \cdots q_{u_n} \le q_{v_1} \cdots q_{v_{|\bv|}}$.
Symmetrically, we have $m \le |\bu|$, hence $q_{v_1} \cdots q_{v_m} \le q_{u_1} \cdots q_{u_{|\bu|}}$.
By Lemma~\ref{l:phi}, we have 
\begin{equation} \label{e:dimuvvu}
\dim_H \pi_{q_0,q_1}\big(\Omega_{\bu\bv^\infty,\bv\bu^\infty} \setminus \Omega_{\bu^\infty,\bv^\infty}\big) \le \dim_H \pi_{q_0,q_1}\big(\{\bu,\bv\}^\infty\big) = \min\{t,1\}, 
\end{equation}
with
\[
\frac{1}{q_{u_1}^t\cdots q_{u_{|\bu|}}^t} + \frac{1}{q_{v_1}^t\cdots q_{v_{|\bv|}}^t} = 1.
\]
Then $q_{u_1} \cdots q_{u_n} \le q_{v_1} \cdots q_{v_{|\bv|}}$ and $q_{v_1} \cdots q_{v_m} \le q_{u_1} \cdots q_{u_{|\bu|}}$ imply that $s \ge t$, thus \eqref{e:dimuvvu} and \eqref{e:dimuv} give that  $d_{q_0,q_1}(\bu^\infty,\bv^\infty) = d_{q_0,q_1}(\bu\bv^\infty,\bv\bu^\infty)$.
\end{proof}

\begin{proof}[Proof of Theorem~\ref{t:equaldim}]
We show first \eqref{e:equaldimuv}.
By Lemmas~\ref{l:uvtilde} and \ref{l:notuvprime}, it is sufficient to prove the statement for $\bu,\bv$ with $\bu^\infty, \bv^\infty \in \Omega_{\bu^\infty,\bv^\infty}$ and no decomposition $\bu = \tilde{\bu}\tilde{\bv}$, $\bv = \tilde{\bv}\tilde{\bu}$.
The inequality $h(\Omega_{\bu^\infty,\bv^\infty}) > 0$ implies that $\bu \in 01\{0,1\}^*$, $\bv \in 10\{0,1\}^*$, thus the statement follows from Proposition~\ref{p:uv}.

For $d_{q_0,q_1}(\bu^\infty,\bb) = d_{q_0,q_1}(\bu\bb,\bb)$, it is sufficient to note that each sequence in $\Omega_{\bu\bb,\bb} \setminus \Omega_{\bu^\infty,\bb}$ contains $\bu 1$ and ends therefore with $\bu \bb$, cf.\ the proof of Lemma~\ref{l:phi}.
The equality $d_{q_0,q_1}(\ba,\bv^\infty) = d_{q_0,q_1}(\ba,\bv\ba)$ is symmetric.
\end{proof}

\color{black}

\end{document}